\documentclass{amsart}

\title[Arithmetic potentialism]{The modal logic of arithmetic potentialism and the universal algorithm}

\author{Joel David Hamkins}
\address[Joel David Hamkins]
{O'Hara Professor of Logic, University of Notre Dame, 100 Malloy Hall, Notre Dame, IN 46556 USA}
\email{jdhamkins@nd.edu}
\urladdr{http://jdh.hamkins.org}

\thanks{Thanks to Rasmus Blanck, Ali Enayat, Victoria Gitman, Roman Kossak, Volodya Shavrukov, and Albert Visser for insightful comments and helpful discussions.
}

\usepackage{latexsym,amsfonts,amsmath,amssymb,mathrsfs}
\usepackage{relsize}
\usepackage[hidelinks]{hyperref}
\usepackage[rgb,dvipsnames]{xcolor}
\usepackage{tikz}
\usetikzlibrary{arrows,arrows.meta,petri,topaths,positioning,shapes,shapes.misc,patterns,calc,decorations.pathreplacing,hobby}
\usepackage{enumitem}

\newtheorem{theorem}{Theorem}
\newtheorem*{theorem*}{Theorem}

\newtheorem*{maintheorem*}{Main Theorem}
\newtheorem*{maintheorems*}{Main Theorems}
\newtheorem{corollary}[theorem]{Corollary}
\newtheorem*{corollary*}{Corollary}
\newtheorem*{corollaries*}{Corollaries}

\newtheorem{lemma}[theorem]{Lemma}

\newtheorem{observation}[theorem]{Observation}

\theoremstyle{definition}

\newtheorem*{definition*}{Definition}

\newtheorem{question}[theorem]{Question}
\newtheorem*{question*}{Question}

\newtheorem*{questions*}{Questions}
\newtheorem*{mainquestion*}{Main Question} 
\newtheorem*{openquestion*}{Open Question} 
\theoremstyle{remark}

\theoremstyle{plain}
\newcommand{\QED}{\end{proof}}

\def\proclaim[#1]{{\bf #1}}
\def\BF#1.{{\bf #1.}}

\def\says#1:#2\par{\item[#1] #2\par}

\newcommand{\Godel}{G\"odel}

\newcommand{\Lowe}{L\"owe}

\newcommand{\Oystein}{{\O}ystein}

\newcommand{\N}{{\mathbb N}}

\makeatletter
\newcommand{\dotminus}{\mathbin{\text{\@dotminus}}}
\newcommand{\@dotminus}{%
  \ooalign{\hidewidth\raise1ex\hbox{.}\hidewidth\cr$\m@th-$\cr}%
}
\makeatother

\newcommand{\of}{\subseteq}

\newcommand{\ofneq}{\subsetneq}

\newcommand{\set}[1]{\{\,{#1}\,\}}

\newcommand{\elesub}{\prec}
\newcommand{\eleequiv}{\equiv}

\newcommand{\tp}{\mathop{\rm tp}}

\newcommand{\SSy}{\mathop{\rm SSy}}

\newcommand{\Con}{\mathop{{\rm Con}}}
\newcommand{\Mod}{\mathop{{\rm Mod}}}

\newcommand{\satisfies}{\models}

\newcommand{\proves}{\vdash}

\DeclareMathOperator{\possible}{\text{\tikz[scale=.6ex/1cm,baseline=-.6ex,rotate=45,line width=.1ex]{\draw (-1,-1) rectangle (1,1);}}}
\DeclareMathOperator{\necessary}{\text{\tikz[scale=.6ex/1cm,baseline=-.6ex,line width=.1ex]{\draw (-1,-1) rectangle (1,1);}}}

\DeclareMathOperator{\uppossible}{\text{\tikz[scale=.6ex/1cm,baseline=-.6ex,rotate=45,line width=.1ex]{
                            \draw (-1,-1) rectangle (1,1); \draw[very thin] (-1,.6) -- (.6,.6) -- (.6,-1);}}}
\DeclareMathOperator{\upnecessary}{\text{\tikz[scale=.6ex/1cm,baseline=-.6ex,line width=.1ex]{
                            \draw (-1,-1) rectangle (1,1); \draw[very thin] (-1,.6) -- (.6,.6) -- (.6,-1);}}}
\DeclareMathOperator{\xpossible}{\text{\tikz[scale=.6ex/1cm,baseline=-.6ex,rotate=45,line width=.1ex]{
                            \draw (-1,-1) rectangle (1,1); \draw[very thin] (-.6,-.6) rectangle (.6,.6);}}}
\DeclareMathOperator{\xnecessary}{\text{\tikz[scale=.6ex/1cm,baseline=-.6ex,line width=.1ex]{
                            \draw (-1,-1) rectangle (1,1); \draw[very thin] (-.6,-.6) rectangle (.6,.6);}}}

\DeclareMathOperator{\soliduppossible}{\text{\tikz[scale=.6ex/1cm,baseline=-.6ex,rotate=45,line width=.1ex]{
                            \draw (-1,-1) rectangle (1,1); \draw[very thin,fill=gray,fill opacity=.25] (-1,-1) rectangle (.6,.6);}}}
\DeclareMathOperator{\solidupnecessary}{\text{\tikz[scale=.6ex/1cm,baseline=-.6ex,line width=.1ex]{
                            \draw (-1,-1) rectangle (1,1); \draw[very thin,fill=gray,fill opacity=.25] (-1,-1) rectangle (.6,.6);}}}
\DeclareMathOperator{\solidxpossible}{\text{\tikz[scale=.6ex/1cm,baseline=-.6ex,rotate=45,line width=.1ex]{
                            \draw (-1,-1) rectangle (1,1); \draw[very thin,fill=gray,fill opacity=.25] (-.6,-.6) rectangle (.6,.6);}}}
\DeclareMathOperator{\solidxnecessary}{\text{\tikz[scale=.6ex/1cm,baseline=-.6ex,line width=.1ex]{
                            \draw (-1,-1) rectangle (1,1); \draw[very thin,fill=gray,fill opacity=.25] (-.6,-.6) rectangle (.6,.6);}}}
\DeclareMathOperator{\consuppossible}{\text{\tikz[scale=.6ex/1cm,baseline=-.6ex,rotate=45,line width=.1ex]{
                            \draw (-1,-1) rectangle (1,1); \draw[very thin] (-1,-1) rectangle (.6,.6);
                            \clip (-1,-1) rectangle (.6,.6); 
                            \draw[very thin] (-1,-1) -- (.6,.6);}}}
\DeclareMathOperator{\consupnecessary}{\text{\tikz[scale=.6ex/1cm,baseline=-.6ex,line width=.1ex]{
                            \draw (-1,-1) rectangle (1,1); \draw[very thin] (-1,-1) rectangle (.6,.6);
                            \clip (-1,-1) rectangle (.6,.6); 
                            \draw[very thin] (-1,-1) -- (.6,.6);}}}
\DeclareMathOperator{\consxpossible}{\text{\tikz[scale=.6ex/1cm,baseline=-.6ex,rotate=45,line width=.1ex]{
                            \draw (-1,-1) rectangle (1,1); \draw[very thin] (-.6,-.6) rectangle (.6,.6);
                            \clip (-1,-1) rectangle (.6,.6); 
                            \draw[very thin] (-.6,-.6) -- (.6,.6);}}}
\DeclareMathOperator{\consxnecessary}{\text{\tikz[scale=.6ex/1cm,baseline=-.6ex,line width=.1ex]{
                            \draw (-1,-1) rectangle (1,1); \draw[very thin] (-.6,-.6) rectangle (.6,.6);
                            \clip (-.6,-.6) rectangle (.6,.6); 
                            \draw[very thin] (-.6,-.6) -- (.6,.6);}}}
\DeclareMathOperator{\solidconsuppossible}{\text{\tikz[scale=.6ex/1cm,baseline=-.6ex,rotate=45,line width=.1ex]{
                            \draw (-1,-1) rectangle (1,1); \draw[very thin,fill=gray,fill opacity=.25] (-1,-1) rectangle (.6,.6);
                            \draw[very thin] (-1,-1) rectangle (.6,.6);
                            \clip (-1,-1) rectangle (.6,.6); 
                            \draw[very thin] (-1,-1) -- (.6,.6);}}}
\DeclareMathOperator{\solidconsupnecessary}{\text{\tikz[scale=.6ex/1cm,baseline=-.6ex,line width=.1ex]{
                            \draw (-1,-1) rectangle (1,1); \draw[very thin,fill=gray,fill opacity=.25] (-1,-1) rectangle (.6,.6);
                            \draw[very thin] (-1,-1) rectangle (.6,.6);
                            \clip (-1,-1) rectangle (.6,.6); 
                            \draw[very thin] (-1,-1) -- (.6,.6);}}}
\DeclareMathOperator{\solidconsxpossible}{\text{\tikz[scale=.6ex/1cm,baseline=-.6ex,rotate=45,line width=.1ex]{
                            \draw (-1,-1) rectangle (1,1); \draw[very thin,fill=gray,fill opacity=.25] (-.6,-.6) rectangle (.6,.6);
                            \draw[very thin] (-.6,-.6) rectangle (.6,.6);
                            \clip (-1,-1) rectangle (.6,.6); 
                            \draw[very thin] (-.6,-.6) -- (.6,.6);}}}
\DeclareMathOperator{\solidconsxnecessary}{\text{\tikz[scale=.6ex/1cm,baseline=-.6ex,line width=.1ex]{
                            \draw (-1,-1) rectangle (1,1); \draw[very thin,fill=gray,fill opacity=.25] (-.6,-.6) rectangle (.6,.6);
                            \draw[very thin] (-.6,-.6) rectangle (.6,.6);
                            \clip (-1,-1) rectangle (.6,.6); 
                            \draw[very thin] (-.6,-.6) -- (.6,.6);}}}
\DeclareMathOperator{\ssypossible}{\text{\tikz[scale=.6ex/1cm,baseline=-.6ex,rotate=45,line width=.1ex]{
                            \draw (-1,-1) rectangle (1,1); \draw[very thin] (-1,-1) rectangle (.6,.6);
                            \draw[very thin] (-.6,-.6) rectangle (.2,.2);
                            }}}
\DeclareMathOperator{\ssynecessary}{\text{\tikz[scale=.6ex/1cm,baseline=-.6ex,line width=.1ex]{
                            \draw (-1,-1) rectangle (1,1); \draw[very thin] (-1,-1) rectangle (.6,.6);
                            \draw[very thin] (-.6,-.6) rectangle (.2,.2);
                            }}}
\DeclareMathOperator{\solidssypossible}{\text{\tikz[scale=.6ex/1cm,baseline=-.6ex,rotate=45,line width=.1ex]{
                            \draw (-1,-1) rectangle (1,1); \draw[very thin] (-1,-1) rectangle (.6,.6);
                            \draw[very thin,fill=gray,fill opacity=.25] (-.6,-.6) rectangle (.2,.2);
                            }}}
\DeclareMathOperator{\solidssynecessary}{\text{\tikz[scale=.6ex/1cm,baseline=-.6ex,line width=.1ex]{
                            \draw (-1,-1) rectangle (1,1); \draw[very thin] (-1,-1) rectangle (.6,.6);
                            \draw[very thin,fill=gray,fill opacity=.25] (-.6,-.6) rectangle (.2,.2);
                            }}}
\DeclareMathOperator{\ipossible}{\text{\tikz[scale=.6ex/1cm,baseline=-.6ex,rotate=45,line width=.1ex]{
                            \draw (-1,-1) rectangle (1,1); \draw[very thin] (-1,-1) rectangle (.6,.6);
                            \draw[very thin] (.2,-1) arc (0:90:1.2);
                            }}}
\DeclareMathOperator{\inecessary}{\text{\tikz[scale=.6ex/1cm,baseline=-.6ex,line width=.1ex]{
                            \draw (-1,-1) rectangle (1,1); \draw[very thin] (-1,-1) rectangle (.6,.6);
                            \draw[very thin] (.2,-1) arc (0:90:1.2);
                            }}}
\DeclareMathOperator{\solidipossible}{\text{\tikz[scale=.6ex/1cm,baseline=-.6ex,rotate=45,line width=.1ex]{
                            \draw (-1,-1) rectangle (1,1); \draw[very thin] (-1,-1) rectangle (.6,.6);
                            \draw[very thin,fill=gray,fill opacity=.25] (-1,-1) -- (.2,-1) arc (0:90:1.2) -- cycle;
                            }}}
\DeclareMathOperator{\solidinecessary}{\text{\tikz[scale=.6ex/1cm,baseline=-.6ex,line width=.1ex]{
                            \draw (-1,-1) rectangle (1,1); \draw[very thin] (-1,-1) rectangle (.6,.6);
                            \draw[very thin,fill=gray,fill opacity=.25] (-1,-1) -- (.2,-1) arc (0:90:1.2) -- cycle ;
                            }}}

\newcommand\dbrace{\hskip-1.5em\raise3pt\hbox{\rotatebox[origin=c]{-35}{$\left.\strut^{\phantom{|}}\right\}$}}}

\newcommand\UParroW{{\setbox0\hbox{$\Uparrow$}\rlap{\hbox to \wd0{\hss$\mid$\hss}}\box0}}

\newcommand{\theoryf}[1]{{\rm #1}}

\newcommand{\concat}{\mathbin{{}^\smallfrown}}
\newcommand{\converges}{{\downarrow}}

\renewcommand{\setminus}{\raise.3ex\hbox{\rotatebox{-20}{$-$}}} 

\newcommand{\intersect}{\cap}

\newcommand{\smalllt}{\mathrel{\mathchoice{\raise2pt\hbox{$\scriptstyle<$}}{\raise1pt\hbox{$\scriptstyle<$}}{\raise0pt\hbox{$\scriptscriptstyle<$}}{\scriptscriptstyle<}}}
\newcommand{\smallleq}{\mathrel{\mathchoice{\raise2pt\hbox{$\scriptstyle\leq$}}{\raise1pt\hbox{$\scriptstyle\leq$}}{\raise1pt\hbox{$\scriptscriptstyle\leq$}}{\scriptscriptstyle\leq}}}

   \def\DHLhksqrt#1#2{%
   \setbox0=\hbox{$#1\sqrt{#2\,}$}\dimen0=\ht0
   \advance\dimen0-0.2\ht0
   \setbox2=\hbox{\vrule height\ht0 depth -\dimen0}%
   {\box0\lower0.4pt\box2}}

\def\[#1]{\mathopen{\lbrack\!\lbrack}#1\mathclose{\rbrack\!\rbrack}}
\newbox\gnBoxA
\newbox\gnBoxB
\newdimen\gnCornerHgt
\setbox\gnBoxA=\hbox{\tiny$\ulcorner$}
\global\gnCornerHgt=\ht\gnBoxA
\newdimen\gnArgHgt
\def\gcode #1{%
\setbox\gnBoxA=\hbox{$#1$}%
\setbox\gnBoxB=\hbox{$\bar #1$}%
\gnArgHgt=\ht\gnBoxB%
\ifnum     \gnArgHgt<\gnCornerHgt \gnArgHgt=0pt%
\else \advance \gnArgHgt by -\gnCornerHgt%
\fi \raise\gnArgHgt\hbox{\tiny$\ulcorner$} \box\gnBoxA %
\raise\gnArgHgt\hbox{\tiny$\urcorner$}}
\newcommand{\UnderTilde}[1]{{\setbox1=\hbox{$#1$}\baselineskip=0pt\vtop{\hbox{$#1$}\hbox to\wd1{\hfil$\sim$\hfil}}}{}}
\newcommand{\Undertilde}[1]{{\setbox1=\hbox{$#1$}\baselineskip=0pt\vtop{\hbox{$#1$}\hbox to\wd1{\hfil$\scriptstyle\sim$\hfil}}}{}}
\newcommand{\undertilde}[1]{{\setbox1=\hbox{$#1$}\baselineskip=0pt\vtop{\hbox{$#1$}\hbox to\wd1{\hfil$\scriptscriptstyle\sim$\hfil}}}{}}
\newcommand{\UnderdTilde}[1]{{\setbox1=\hbox{$#1$}\baselineskip=0pt\vtop{\hbox{$#1$}\hbox to\wd1{\hfil$\approx$\hfil}}}{}}
\newcommand{\Underdtilde}[1]{{\setbox1=\hbox{$#1$}\baselineskip=0pt\vtop{\hbox{$#1$}\hbox to\wd1{\hfil\scriptsize$\approx$\hfil}}}{}}

\renewcommand{\implies}{\mathrel{\rightarrow}}

\renewcommand{\iff}{\mathrel{\leftrightarrow}}
\newcommand{\Iff}{\mathrel{\Longleftrightarrow}}

\def\<#1>{\left\langle#1\right\rangle}

\newcommand{\Tr}{\mathop{\rm Tr}\nolimits}

\newcommand{\ZFC}{{\rm ZFC}}

\newcommand{\PA}{{\rm PA}}

\newcommand{\cell}[1]{\boxit{\hbox to 17pt{\strut\hfil$#1$\hfil}}}
\newcommand{\head}[2]{\lower2pt\vbox{\hbox{\strut\footnotesize\it\hskip3pt#2}\boxit{\cell#1}}}
\newcommand{\boxit}[1]{\setbox4=\hbox{\kern2pt#1\kern2pt}\hbox{\vrule\vbox{\hrule\kern2pt\box4\kern2pt\hrule}\vrule}}
\newcommand{\Col}[3]{\hbox{\vbox{\baselineskip=0pt\parskip=0pt\cell#1\cell#2\cell#3}}}
\newcommand{\tapenames}{\raise 5pt\vbox to .7in{\hbox to .8in{\it\hfill input: \strut}\vfill\hbox to
.8in{\it\hfill scratch: \strut}\vfill\hbox to .8in{\it\hfill output: \strut}}}
\newcommand{\Head}[4]{\lower2pt\vbox{\hbox to25pt{\strut\footnotesize\it\hfill#4\hfill}\boxit{\Col#1#2#3}}}
\newcommand{\Dots}{\raise 5pt\vbox to .7in{\hbox{\ $\cdots$\strut}\vfill\hbox{\ $\cdots$\strut}\vfill\hbox{\
$\cdots$\strut}}}
\newcommand{\df}{\it} 
\usepackage[backend=bibtex,style=alphabetic,maxbibnames=15,maxcitenames=6,dateabbrev=false]{biblatex}
\renewcommand{\UrlFont}{}
\renewbibmacro{in:}{\ifentrytype{article}{}{\printtext{\bibstring{in}\intitlepunct}}} 
\DeclareFieldFormat{url}{{\UrlFont\url{#1}}} 
\DeclareFieldFormat{urldate}{
  (version \thefield{urlday}\addspace%
  \mkbibmonth{\thefield{urlmonth}}\addspace%
  \thefield{urlyear}\isdot)}
\DeclareFieldFormat{eprint:arxiv}{
  \ifhyperref
    {\href{http://arxiv.org/abs/#1}{%
        arXiv\addcolon\nolinkurl{#1}}\iffieldundef{eprintclass}{}{\UrlFont{\mkbibbrackets{\thefield{eprintclass}}}}}
    {arXiv\addcolon\nolinkurl{#1}\iffieldundef{eprintclass}{}{\UrlFont{\mkbibbrackets{\thefield{eprintclass}}}}}}
\bibliography{HamkinsBiblio,MathBiblio,WebPosts,PhilBiblio}
\newcommand\Val{\mathord{\rm Val}}
\newcommand{\PAp}{\PA^{\scriptscriptstyle\!+}}
\newcommand{\setrandomcolor}{%
  \definecolor{randomcolor}{RGB}{\pdfuniformdeviate 256,\pdfuniformdeviate 256,\pdfuniformdeviate 256}%
}
\newcommand{\Larith}{\mathcal{L}}

\begin{document}

\begin{abstract}
I investigate the modal commitments of various conceptions of the philosophy of arithmetic potentialism. Specifically, I shall consider the potentialist conceptions arising from a model-theoretic view of the models of arithmetic as possible arithmetic realms of feasibility, considering them under their natural extension concepts, such as end-extensions, arbitrary extensions, conservative extensions and more, which in effect express distinct potentialist ideas. In these potentialist systems, I show, the propositional modal assertions that are valid with respect to all arithmetic assertions with parameters are exactly the assertions of \theoryf{S4}. With respect to sentences, however, the validities of a model lie between \theoryf{S4} and \theoryf{S5}, and these bounds are sharp in that there are models realizing both endpoints. For a model of arithmetic to validate \theoryf{S5} is precisely to fulfill the arithmetic maximality principle, which asserts that every possibly necessary statement is already true, and these models are equivalently characterized as those satisfying a maximal $\Sigma_1$ theory. The main \theoryf{S4} analysis makes fundamental use of the universal algorithm, of which this article provides a simplified, self-contained account. The main philosophical point is that fundamentally different potentialist conceptions---linear inevitability, convergent potentialism and radical branching possibility---are revealed by the precise modal validities of the corresponding potentialist systems in which those attitudes are expressed, and so it is important to discover them.
\end{abstract}

\maketitle

\section{Introduction}

The philosophy of potentialism, originating in antiquity in the classical dispute between actual and potential infinity, has recently enjoyed a resurgence of interest by philosophers. Current work emphasizes the modal nature of potentialism, finding the essence of potentialism to lie in the accompanying hierarchy of universe fragments, which provide a natural realm for modal analysis and assertions. Thus we separate the core idea of potentialism from its prior connection with infinity, for one can have potentialist conceptions of universe fragments, even when some of those fragments are already infinite and contain completed infinities. In short, according to the new perspective, potentialism is not necessarily about infinity at all, but rather about the idea of a realm of universe fragments with respect to a notion of extension.

\Oystein\ Linnebo and Stewart Shapiro~\cite{Linnebo:2013-PHS, LinneboShapiro2017:Actual-and-potential-infinity} emphasize this modal perspective on potentialism, and there is a growing literature. A vast spectrum of potentialist conceptions is emerging in diverse foundational domains---in set theory, in arithmetic, and in model theory generally for any kind of mathematical structure and any first-order theory. In \cite{HamkinsLinnebo2022:Modal-logic-of-set-theoretic-potentialism}, Linnebo and I explored the exact modal commitments of various kinds of convergent set-theoretic potentialism, including set-theoretic rank potentialism (true in all larger $V_\beta$); Grothendieck-Zermelo potentialism (true in all larger $V_\kappa$ for inaccessible cardinals $\kappa$); transitive-set potentialism (true in all larger transitive sets); countable-transitive-model potentialism (true in all larger countable transitive models of \ZFC); countable-model potentialism (true in all larger countable models of \ZFC); and others. My earlier work with Benedikt L\"owe and George Leibman on the modal logic of forcing \cite{HamkinsLoewe2008:TheModalLogicOfForcing, HamkinsLoewe2013:MovingUpAndDownInTheGenericMultiverse, HamkinsLeibmanLoewe2015:StructuralConnectionsForcingClassAndItsModalLogic} is now naturally seen as fundamentally potentialist in nature. W.~Hugh Woodin and I explored in \cite{HamkinsWoodin:The-universal-finite-set} the modal commitments of (nonconvergent) top-extensional set-theoretic potentialism, with further work \cite{HamkinsWilliams2021:The-universal-finite-sequence, Hamkins2024:Every-countable-model-of-arithmetic-or-set-theory-has-a-pointwise-definable-end-extension} on end-extensional set-theoretic potentialism. In \cite{HamkinsWoloszyn2024:Modal-model-theory}, Wojciech Wo\l oszyn and I develop a general theory of modal model theory, injecting modal vocabulary into model-theoretic investigations, with the class $\Mod(T)$ of all models of a given first-order theory $T$ as a principal motivating example, viewed as a potentialist system of possible worlds under the substructure relation. So we now have modal graph theory, modal group theory, modal field theory, and so forth, all fundamentally potentialist in nature. In his dissertation work, Wo\l oszyn has introduced a potentialist modal perspective to any concrete Kripke category, including cases with noninjective morphisms, which are therefore potentialist conceptions violating $x\neq y\implies\necessary x\neq y$.

For potentialist arithmetic, one now naturally views Woodin's work on the universal algorithm \cite{Woodin2011:A-potential-subtlety-concerning-the-distinction-between-determinism-and-nondeterminism} as potentialist in nature, as well as \cite{SavelievShapirovsky2016:On-modal-logic-of-submodels, SavelievShapirovsky2018:On-modal-logics-of-model-theoretic-relations} and \cite{Visser1998:An-overview-of-interpretability-logic}, focused on the modal validities of relative interpretability, including arithmetic extension modalities. In this article, I similarly find the exact modal commitments of various kinds of arithmetic potentialism.

The main lesson to be learned is that the differing potentialist conceptions often exhibit fundamentally different modal validities.
\begin{figure}[h]
\begin{tikzpicture}[scale=.4,xscale=.7,line join=bevel]
\node (0) at (0,0) {};
\node (a) at (0,4) {};
\node (b) at (0,6) {};
\node (c) at (0,7.9) {};
\node (d) at (0,10) {};
 \draw[fill=blue!67!red,fill opacity=.32,thin] (0.center) -- ([xshift=.8cm]a.center) -- node[below,opacity=1,scale=.7] {$M$} ([xshift=-.8cm]a.center) -- cycle;
 \draw[fill=blue!50!red,opacity=.7,fill opacity=.24,thin] ([xshift=.8cm]a.center) -- ([xshift=1.2cm]b.center) -- node[below,opacity=.7,scale=.7] {$M_1$} ([xshift=-1.2cm]b.center) -- ([xshift=-.8cm]a.center);
 \draw[fill=red!80!yellow,dash pattern={on 8pt off 1pt},opacity=.6,fill opacity=.16,thin] ([xshift=1.2cm]b.center) -- ([xshift=1.6cm]c.center) -- node[below,opacity=.6,scale=.7] {$M_2$} ([xshift=-1.6cm]c.center) -- ([xshift=-1.2cm]b.center);
 \draw[fill=yellow,dotted,opacity=.3,fill opacity=.08] ([xshift=1.6cm]c.center) -- ([xshift=2cm]d.center) -- node[below,opacity=.3,scale=.7] {$M_3$} ([xshift=-2cm]d.center) -- ([xshift=-1.6cm]c.center);
 \node[align=center] at (0,-1.5) {Linear\\ inevitability\\ \theoryf{S4.3}};
\begin{scope}[shift={(12,0)}]
\node (0) at (0,0) {};
\node (a) at (-.8,4) {};
\node (b) at (.8,4) {};
\node (c) at (-3,7.7) {};
\node (d) at (-.5,7.7) {};
\node (e) at (.2,6.9) {};
\node (f) at (2.5,6.9) {};
 \draw[fill=blue!67!red,fill opacity=.24,thin] (0.center) -- (a.center) -- node[below,opacity=1,scale=.7] {$M$} (b.center) -- cycle;
 \draw[fill=blue,opacity=.8,dash pattern={on 7pt off 1pt}, fill opacity=.08] (a.center) to[out=104,in=-60] (c.center) -- node[below,opacity=.8,scale=.6] {$M'$} (d.center) to[out=-85,in=76] (b.center);
 \draw[fill=red,opacity=.7,fill opacity=.08,thin] (a.center) to[out=104,in=-100] (e.center) -- node[below,opacity=.7,scale=.6] {$M''$} (f.center) to[out=-120,in=76] (b.center);
 \draw[fill=yellow,dotted,opacity=.4,fill opacity=.08] (0.center) to[out=102,in=-65] ([shift={(-1,1)}]c.center) -- node[below,opacity=.4,scale=.7] {$N$} ([shift={(4,1)}]d.center) to[out=-110,in=78] cycle;
 \node[align=center] at (0,-1.5)  {Directed\\ convergence\\ \theoryf{S4.2}};
\end{scope}
\begin{scope}[shift={(24,0)}]
\node (0) at (0,0) {};
\node (a) at (-1,4) {};
\node (b) at (1,4) {};
\node (c) at (-4,7.1) {};
\node (d) at (-1,7.1) {};
\node (e) at (0,6.2) {};
\node (f) at (2.5,6.2) {};
\node (g) at (2.1,9.5) {};
\node (h) at (5.1,9.5) {};
\node (i) at (-2,10) {};
\node (j) at (1.5,10) {};
 \draw[fill=blue,fill opacity=.16,thin] (0.center) -- (a.center) -- node[below,opacity=1,scale=.7] {$M$} (b.center) -- cycle;
 \draw[fill=blue,opacity=.9,fill opacity=.08] (0.center) -- (a.center) to[out=104,in=-60] (c.center) -- node[below,opacity=.9,scale=.6] {$M_0$} (d.center) to[out=-85,in=76] (b.center) -- cycle;
 \draw[fill=red!80!yellow,opacity=.8,fill opacity=.24,thin] (0.center) -- (a.center) to[out=104,in=-135] (e.center) -- node[below,opacity=.8,scale=.6] {$M_1$} (f.center) to[out=-135,in=76] (b.center) -- cycle;
 \draw[fill=red,opacity=.7,dash pattern={on 8pt off 1pt},fill opacity=.08] (e.center) to[out=45,in=-84] (g.center) -- node[below,opacity=.7,scale=.6] {$M_{11}$} (h.center) to[out=-125,in=45] (f.center);
 \draw[fill=yellow,opacity=.5,dotted,fill opacity=.08] (e.center) to[out=45,in=-75] (i.center) -- node[below,opacity=.5,scale=.6] {$M_{10}$} (j.center) to[out=-104,in=45] (f.center);
 \node[align=center] at (0,-1.5)  {Radical\\ branching\\ \theoryf{S4}};
\end{scope}
\end{tikzpicture}
\caption{Differing potentialist conceptions}
\end{figure}
The Grothendieck-Zermelo set-theoretic potentialism analyzed in~\cite{HamkinsLinnebo2022:Modal-logic-of-set-theoretic-potentialism}, for example, with universe fragments being the inaccessible rank-initial segments $V_\kappa$ of the set-theoretic universe, validates exactly \theoryf{S4.3}; the forcing potentialism of~\cite{HamkinsLoewe2008:TheModalLogicOfForcing} validates exactly \theoryf{S4.2}; and the top-extensional set-theoretic potentialism analyzed in~\cite{HamkinsWoodin:The-universal-finite-set} validates exactly~\theoryf{S4}.

These different modal theories reflect the fundamentally different character of potentiality offered by these various potentialist systems. One generally has \theoryf{S4.3} in a potentialist system when the worlds are linearly ordered; \theoryf{S4.2} when they are convergent; and \theoryf{S4} when they have the character of a tree, with branching possibility. In each case, it is usually easy to verify that the given theory is valid; the far more difficult, subtle results are in showing for these systems that no modal assertions beyond these theories is valid. In short, lower bounds are cheap; upper bounds are difficult.

The broad philosophical point to be made here is that a satisfactory modal account of potentialism must now address these fundamentally different kinds of potentiality. It is no longer adequate merely to present a naive potentialist account of arithmetic potentialism, asserting perhaps that at any moment one has only some of the natural numbers but that one can always have more as time proceeds. Such an account misses what we now see as the key points of contention, such as the central dichotomy of convergence versus branching possibility. Are we to expect the universe as it unfolds to have a character of linear inevitability, where the numbers pile on in a unique coherent manner, converging to an ultimate limit model? Or shall we instead expect the unexpected, with the universe unfolding in a process of radical branching possibility? These notions of potentiality have fundamentally different characters. To my way of thinking, one of the important philosophical contributions made by the model-theoretic analysis is that the branching possibility potentialist systems, such as the ones considered in this paper, show how the radical branching potentialist perspective can exhibit a fundamental coherence and resilience.

In order to shed light on these fundamentally philosophical questions, we must therefore achieve a better grasp on the nature of the various potentialist conceptions. And so the project becomes in part mathematical. Indeed, this paper illustrates what I find to be an enjoyable common pattern of exchange between philosophy and mathematics, by which a philosophical idea inspires a mathematical analysis, which in turn raises further philosophical issues, and so on in a fruitful cycle. Let's get started.

\section{Models of arithmetic as potentialist systems}\label{Section.Models-of-arithmetic-as-potentialist-systems}

With the philosophical aim of illuminating the diverse natures of arithmetic potentialism, therefore, I propose that we undertake the comparatively mathematical task of analyzing the models of arithmetic $M$ under the modalities arising from their various natural extension concepts. We may view the result in each case as a potentialist system, a Kripke model of possible arithmetic worlds, exhibiting one of the flavors of potentialism.

The end-extension modality, for example, is defined by the operators:
\begin{align*}
  &M\satisfies\uppossible\varphi&\text{ if and only if }\quad &\varphi\text{ holds in some end-extension of }M,\text{ and}&\\
  &M\satisfies\upnecessary\varphi&\text{ if and only if }\quad &\varphi\text{ holds in all end-extensions of }M. &
\end{align*}
The arbitrary-extension modality, in contrast, is defined by:
\begin{align*}
  &M\satisfies\xpossible\varphi&\text{ if and only if }\quad &\varphi\text{ holds in some extension of }M,\text{ and} &\\
  &M\satisfies\xnecessary\varphi&\text{ if and only if }\quad &\varphi\text{ holds in all extensions of }M. &
\end{align*}
\enlargethispage{20pt}%
Other natural extension modalities include $\uppossible_n$ and $\xpossible_n$, which restrict the previous to $\Sigma_n$-elementary extensions, and in section~\ref{Section.Other-forms-of-arithmetic-potentialism} we shall consider the conservative end-extension modality $\consuppossible$, the computably saturated end-extension modality $\soliduppossible$, and many others.\goodbreak

Ultimately, we shall consider a large variety of natural extension modalities in this article (see page \pageref{Page.List-of-modalities} for definitions):
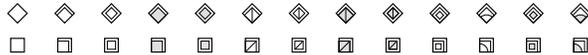
$$\begin{array}{ccccccccccccc}
  \possible & \uppossible & \xpossible & \soliduppossible & \solidxpossible & \consuppossible & \consxpossible & \solidconsuppossible & \solidconsxpossible & \ssypossible & \ipossible & \solidssypossible & \solidipossible\\
  \necessary & \upnecessary & \xnecessary & \solidupnecessary & \solidxnecessary & \consupnecessary & \consxnecessary &
  \solidconsupnecessary & \solidconsxnecessary & \ssynecessary & \inecessary & \solidssynecessary & \solidinecessary \\
\end{array}$$
Each extension concept for the models of arithmetic---a way of relating some models of arithmetic to other larger models---gives rise to a corresponding extension modality and its associated potentialist system of arithmetic.

Each of these extension concepts gives rise to a distinct particular philosophy of arithmetic modal potentialism, by which we view the models of arithmetic as the possible worlds in a suitable Kripke model expressing that particular kind of potentiality. The goal is to analyze the precise modal validities exhibited by these various natural potentialist systems. For generality, let us work in the class of all models of $\PAp$, an arbitrary fixed consistent c.e.~extension of \PA, but the central case is that $\PAp$ might just be \PA\ itself, and I shall use the phrase \emph{model of arithmetic} to mean a model of this fixed base theory $\PAp$. Since the domains of the models are inflationary with respect to these various extensions, these are potentialist systems in the sense of~\cite{HamkinsLinnebo2022:Modal-logic-of-set-theoretic-potentialism}. In that article, Linnebo and I had provided a general model-theoretic framework for potentialism, using it to analyze the precise modal commitments of various kinds of set-theoretic potentialism, and the project here should be seen as an arithmetic analogue. Following this article, Woodin and I carried out an analogous project~\cite{HamkinsWoodin:The-universal-finite-set} for top-extensional potentialism in set theory, with related further work in \cite{HamkinsWilliams2021:The-universal-finite-sequence, Hamkins2024:Every-countable-model-of-arithmetic-or-set-theory-has-a-pointwise-definable-end-extension}.

A main result, proved in section~\ref{Section.Modal-logic-of-arithmetic-potentialism}, will be the following.

\goodbreak
\begin{maintheorem*} With respect to the potentialist system consisting of the models of $\PAp$ under the end-extension modality $\uppossible$:
 \begin{enumerate}
   \item The potentialist validities of any $M\satisfies\PAp$, with respect to arithmetic assertions with parameters from $M$---and indeed just one specific parameter suffices---are exactly the modal assertions of \theoryf{S4}.
   \item The potentialist validities of any $M\satisfies\PAp$, with respect to arithmetic sentences, is a modal theory containing \theoryf{S4} and contained in \theoryf{S5}.
   \item Both of the bounds in (2) are sharp: there are models $M$ validating exactly \theoryf{S4} and others validating exactly \theoryf{S5} for sentences.
 \end{enumerate}
\end{maintheorem*}

I shall prove similar results for the other modalities, including $\uppossible$, $\xpossible$, $\uppossible_n$, $\xpossible_n$, $\consuppossible$, $\soliduppossible$ and sections~\ref{Section.Modal-logic-of-arithmetic-potentialism} and~\ref{Section.Other-forms-of-arithmetic-potentialism} have various further sharper results, including an analysis of the arithmetic maximality principle for these various extension concepts.

This theorem in part refines an earlier independent result of Volodya Shavrukov (appearing in Visser's overview of interpretability logic~\cite[theorem 16, credited to Shavrukov]{Visser1998:An-overview-of-interpretability-logic}), showing that the propositional modal assertions valid for sentential substitutions simultaneously in all models of arithmetic is \theoryf{S4}. Thus, there is a certain overlap of themes and ideas between this paper and prior work of Visser and Shavrukov on the modal logic of interpretability, and also with work of Berarducci, Blanck, Enayat, Japaridze, Shavrukov, Visser and Woodin in connection with the universal algorithm, as explained in section~\ref{Section.Universal-algorithm}.

To review the terminology of the theorem, a modal assertion $\varphi(p_0,\ldots,p_n)$ with propositional variables $p_i$ is said to be \emph{valid} at a world $M$ for a family of assertions, if $M\satisfies\varphi(\psi_0,\ldots,\psi_n)$ for any substitution of the propositional variables by those assertions $p_i\mapsto \psi_i$. Following~\cite{HamkinsLinnebo2022:Modal-logic-of-set-theoretic-potentialism}, let us denote by $\Val(M,\mathcal{L})$ the collection of propositional modal validities of a world $M$ with respect to assertions in a language $\mathcal{L}$. I shall at times use a subscript, as in $\Val_{\uppossible}(M,\Larith)$, to indicate which modality and accessibility relation is intended.

Let me illustrate the idea by showing that the modal assertions of \theoryf{S4} are valid in every model of arithmetic under the end-extension modality $\uppossible$. This is an elementary exercise in modal reasoning, which I encourage the reader to undertake. Every model of arithmetic $M$ obeys $\upnecessary(\varphi\to\psi)\to(\upnecessary\varphi\to\upnecessary\psi)$, since if an implication $\varphi\to\psi$ holds in every end-extension of $M$ and also the hypothesis $\varphi$, then so does the conclusion $\psi$. Similarly, every model obeys $\upnecessary\varphi\to\varphi$, since every model is an end-extension of itself. Every model obeys $\upnecessary\varphi\to\upnecessary\upnecessary\varphi$, since if $\varphi$ holds in all end-extensions, then $\upnecessary\varphi$ also holds in all end-extensions, because an end-extension of an end-extension is an end-extension. And every model obeys the duality axiom $\neg\upnecessary\varphi\iff\uppossible\neg\varphi$, since $\varphi$ fails to hold in all end-extensions just in case $\neg\varphi$ holds in some end-extension. Since these modal statements are precisely the axioms of \theoryf{S4}, to be closed under modus ponens and necessitation, one concludes that \theoryf{S4} is valid.

It is a more severe requirement on a modal assertion $\varphi(p_0,\ldots,p_n)$, of course, for it to be valid with respect to a larger collection of substitution instances; and so in particular, a modal assertion might be valid with respect to sentential substitutions, but not valid with respect to all substitutions by formulas with parameters. For example, the main theorem shows that some models validate \theoryf{S5} for sentences, but only \theoryf{S4} for arithmetic assertions with parameters.

It follows from the main theorem that the standard model $\N$ validates exactly \theoryf{S4} for sentential substitutions, since all parameters are absolutely definable in the standard model. Meanwhile, the models whose sentential validities are \theoryf{S5} are precisely the models of what I call the arithmetic maximality principle, which holds in a model of arithmetic $M$ when $M\satisfies\uppossible\upnecessary\sigma\implies\sigma$ for every arithmetic sentence $\sigma$. In other words, if a sentence $\sigma$ could become true in some end-extension of $M$ in such a way that it remains true in all further end-extensions, then it was already true in $M$. This holds, I prove in theorem~\ref{Theorem.Maximal-Sigma1-extension}, in precisely the models of arithmetic whose $\Sigma_1$ theory is a maximal consistent $\Sigma_1$ extension of the base theory $\PAp$.

Let me clarify the various formal languages that will be used in this work.
\begin{enumerate}
  \item The \emph{language of arithmetic} $\Larith=\set{+,\cdot,0,1,<}$ is the usual language in which for example the theory \PA\ is expressed, with expressions such as
 $$\forall x\exists y\, \left[(x=y+y)\vee(x=(y+y)+1)\right].$$
  \item The \emph{language of propositional logic}, in contrast, denoted $\mathcal{P}$, has no arithmetic or quantifiers, but is simply the closure of propositional variables $p$, $q$, $r$ and so on under Boolean logical connectives, with expressions such as
 $$\neg p\vee(p\implies q).$$
    \item The \emph{language of propositional modal logic}, denoted $\mathcal{P}^{\possible}$, extends $\mathcal{P}$ with the modal operators $\possible,\necessary$, having expressions such as
 $$p\implies\possible(p\wedge\possible\necessary\neg p).$$
\enlargethispage{20pt}\label{Languages}
\item The (partial) \emph{potentialist language of arithmetic}, denoted ${\possible}\Larith$, is the closure of the language of arithmetic $\Larith$ under the modal operators $\possible,\necessary$ and under the Boolean connectives. Thus, every ${\possible}\Larith$ assertion is a substitution instance $\varphi(\psi_0,\ldots,\psi_n)$ of a propositional modal logic assertion $\varphi(p_0,\ldots,p_n)\in\mathcal{P}^{\possible}$ by arithmetic assertions $\psi_i\in\Larith$, such as in the case
 $$\Con(\PA)\implies\possible\left[\Con(\PA)\wedge\possible\necessary \neg\Con(\PA)\right],$$
which is a substitution instance of the propositional modal assertion mentioned in (3) using $\psi=\Con(\PA)$.
 \goodbreak
 \item Finally, the full \emph{potentialist language of arithmetic}, denoted $\Larith^{\possible}$, allows all the usual constructions of the language of arithmetic plus the modal operators, as in the assertion
 $$\forall k \left[\necessary\Con(\PA_k)\implies\necessary\Con(\PA_{k+1})\right].$$
Note that in the full potentialist language of arithmetic, the modal operators can fall under the scope of a number quantifier, while this does not occur for expressions in the partial potentialist language ${\possible}\Larith$.
\end{enumerate}

The paper concludes in section~\ref{Section.Philosophical-remarks} with philosophical remarks on some fundamentally different potentialist attitudes and how they are expressed in the potentialist modal validities.

\section{Extensions of models of arithmetic}

For the rest of this article, let $\PAp$ be a fixed consistent computably enumerable theory extending $\PA$ in the language of arithmetic, defined by a fixed computable enumeration algorithm, which can therefore be interpreted in any model of arithmetic. Consider the potentialist systems consisting of all the models of $\PAp$ under the various extension relations. Let $\PAp_k$ be the theory fragment consisting of the axioms of $\PAp$ enumerated by stage $k$. These assertions will have length less than $k$, and so their arithmetic complexity will be less than $\Sigma_k$.

My argument will use the following classical result, a generalization of the fact that $\PA$ proves $\Con(\PA_k)$ for every standard finite $k$. This fact in turn implies that \PA, if consistent, is not finitely axiomatizable, for it proves the consistency of any of its particular finite fragments.

\begin{theorem}[Mostowski's reflection theorem~\cite{Mostowski1952:On-models-of-axiomatic-systems}]\label{Theorem.Mostowski-reflection-theorem}
For any standard $k$, the theory \PA\ proves $\Con(\Tr_k)$, where $\Tr_k$ is the $\Sigma_k$-definable collection of true $\Sigma_k$ assertions.
\end{theorem}

This is a theorem scheme, with a separate statement for each $k$. In any model of arithmetic $M\satisfies\PA$, the interpretation of $\Tr_k$ is the set of (possibly nonstandard) $\Sigma_k$ statements true in $M$ according to the universal $\Sigma_k$ truth predicate, which is a definable class of $M$. In nonstandard models, of course, this theory can differ significantly from the $\Sigma_k$ theory of the standard model. Note also that even $\Tr_1$ can include many statements that are independent of \PA. It follows immediately from the theorem that $\PAp$ proves the consistency of $\PAp_k$ for any standard finite $k$, since $\PAp_k$ will be included in $\Tr_k$ in any model of $\PAp$.

\enlargethispage{20pt}%
\begin{lemma}[Possibility-characterization lemma]\label{Lemma.Possibility-characterization}
 In the potentialist systems consisting of the models of $\PAp$ under end-extensions or arbitrary extensions, the following are equivalent for any model $M\satisfies\PAp$ and any assertion $\varphi(a)$ in the language of arithmetic with parameter $a\in M$.
 \begin{enumerate}
   \item $M\satisfies\uppossible\varphi(a)$. That is, $\varphi(a)$ is true in some $\PAp$ end-extension of $M$.
   \item $M\satisfies\xpossible\varphi(a)$. That is, $\varphi(a)$ is true in some $\PAp$ extension of $M$.
   \item $\varphi(a)$ is consistent with $\PAp+\tp_{\Sigma_1}^M\!(a)$, adding the $\Sigma_1$ type of $a$ in $M$.
   \item $M\satisfies\Con(\PAp_k+\varphi(a))$ for all standard numbers $k$.
   \item (For $M$ nonstandard) $M\satisfies\Con(\PAp_k+\varphi(a))$ for some nonstandard $k$.
 \end{enumerate}
\end{lemma}\goodbreak

This result can be seen as a version of the Orey-H\'ajek characterization of relative interpretability; an essentially similar observation is made by Blanck and Enayat~\cite[theorem 6.9]{Blanck2017:Dissertation:Contributions-to-the-metamathematics-of-arithmetic}. In section~\ref{Section.Other-forms-of-arithmetic-potentialism}, we shall see that these statements are also equivalent to $\consuppossible\varphi(a)$, $\consxpossible\varphi(a)$, and in computably saturated models, to $\soliduppossible\varphi(a)$, $\solidxpossible\varphi(a)$, and others. In the remarks before theorem~\ref{Theorem.Maximal-Sigma_n-extension}, we explain the analogue of lemma~\ref{Lemma.Possibility-characterization} for the $\Sigma_n$-elementary modal operators $\uppossible_n$ and $\xpossible_n$.

\begin{proof}
($1\to 2$) Immediate, since end-extensions are extensions.

($2\to 3$) Suppose that $M\satisfies\xpossible\varphi(a)$, so that there is some extension $N$ satisfying $\PAp+\varphi(a)$. Since every extension is $\Delta_0$-elementary, it follows that $\Sigma_1$ statements true in $M$ are preserved to $N$, and so $N$ satisfies every assertion in the $\Sigma_1$ type of $a$ in $M$. In other words, $N\satisfies\PAp+\varphi(a)+\tp_{\Sigma_1}^M\!(a)$, and so that theory is consistent.

($3\to 4$) Suppose that $\varphi(a)$ is consistent with $\PAp+\tp_{\Sigma_1}^M\!(a)$. So there is a model $N\satisfies\PAp+\varphi(a)+\tp_{\Sigma_1}^M\!(a)$, although $N$ may not necessarily extend $M$. (I am using $a$ to refer both to the original element of $M$ and also to the object realizing this type in $N$.) For any standard finite $k$, we have $N\satisfies\PAp_k+\varphi(a)$, and since these have bounded standard-finite complexity, it follows by the Mostowski reflection theorem that $N\satisfies\Con(\PAp_k+\varphi(a))$. So it cannot be that the original model $M$ satisfies $\neg\Con(\PAp_k+\varphi(a))$, for if it were, this would be a $\Sigma_1$ assertion about $a$ that is true in $M$ and therefore part of the type $\tp_{\Sigma_1}^M\!(a)$, which is supposed to be true in $N$, contrary to what was just observed.

($4\iff 5$) For nonstandard $M$, the forward implication follows by overspill; the converse implication is immediate, since the statement becomes stronger as $k$ becomes larger.

($4\to 1$) First consider nonstandard $M$. If $M$ thinks that $\PAp_k+\varphi(a)$ is consistent for some nonstandard $k$, then $M$ can complete this theory to a complete consistent Henkin theory and form the resulting Henkin model, which will be a model of $\PAp$ since all standard instances are included in the nonstandard theory fragment $\PAp_k$. It will also be a model $\varphi(a)$, since this is in the Henkin theory. Since the model is definable in $M$, it follows that it will be an end-extension of $M$, and so $\PAp+\varphi(a)$ will be true in an end-extension of $M$. So $M\satisfies\uppossible\varphi(a)$. In the case that $M$ is the standard model $\N$, then we get full $\Con(\PAp+\varphi(a))$, and the Henkin model again provides the desired end-extension.
\end{proof}

In the case of an arithmetic sentence $\sigma$, with no parameter $a$, then the equivalence of statement (3) should be taken as the assertion that $\sigma$ holds in an extension of a model $M$ if and only if it is consistent with $\PAp$ plus the $\Sigma_1$ theory of $M$.

The following dual version of lemma \ref{Lemma.Possibility-characterization} is an immediate consequence.

\begin{lemma}[Necessity-characterization lemma]\label{Lemma.Necessity-characterization}
 In the potentialist systems consisting of the models of $\PAp$ under end-extensions or arbitrary extensions, the following are equivalent for any  model $M\satisfies\PAp$ and any assertion $\varphi(a)$ in the language of arithmetic with parameter $a\in M$.
 \begin{enumerate}
   \item $M\satisfies\upnecessary\varphi(a)$. That is, $\varphi(a)$ is true in all $\PAp$ end-extensions of $M$.
   \item $M\satisfies\xnecessary\varphi(a)$. That is, $\varphi(a)$ is true in all $\PAp$ extensions of $M$.
   \item $\PAp+\tp_{\Sigma_1}^M\!(a)\proves\varphi(a)$ .
   \item $M\satisfies\neg\Con(\PAp_k+\neg\varphi(a))$ for some standard number $k$.
   \item $M\satisfies\ \PAp_k\proves\varphi(a)$ for some standard finite $k$.
 \end{enumerate}
\end{lemma}

I find it interesting to notice that lemma~\ref{Lemma.Possibility-characterization} shows that possibility assertions such as $\uppossible\varphi$ or $\xpossible\varphi$, where $\varphi$ is arithmetic, have a $\forall$ character inside $M$ rather than $\exists$, because in statements (3) and (4) they are each equivalent to the infinite conjunction of arithmetic assertions over standard $k$, a kind of universal statement, whereas possibility is usually thought of for Kripke models as a kind of existential statement. The point is that model-existence assertions in this context amount to consistency, which is a $\forall$ assertion. Similarly, lemma~\ref{Lemma.Necessity-characterization} shows that all instances of necessity $\upnecessary\varphi(a)$ and $\xnecessary\varphi(a)$ for arithmetic assertions $\varphi(a)$ are caused by $\Sigma_1$ witnessing assertions in $M$, namely, the existence of a proof of $\varphi(a)$ from $\PAp_k$ in $M$ for some standard $k$. Although any true $\Sigma_1$ assertion is necessary, example~\ref{Observation.Necessary-but-not-Sigma1} shows that there are statements that are not $\Sigma_1$ and not provably equivalent to any $\Sigma_1$ assertion, which can nevertheless be necessary over a model of arithmetic. The lemma shows, however, that whenever this happens, it is because they are a provable consequence of a certain $\Sigma_1$ assertion that happens to be true in the model, namely, the inconsistency assertion of (4) or the proof-existence assertion of (5) in lemma~\ref{Lemma.Necessity-characterization}.

Putting the two lemmas together, we arrive at the following.

\begin{theorem}\label{Theorem.Possibly-necessary-characterization}
   In the potentialist systems consisting of the models of $\PAp$ under end-extensions or arbitrary extensions, the following are equivalent for any model $M\satisfies\PAp$ and any formula $\varphi$ in the language of arithmetic with parameter $a\in M$.
   \begin{enumerate}
     \item $M\satisfies\uppossible\upnecessary\varphi(a)$. That is, $\varphi(a)$ is end-extension possibly necessary over $M$.
     \item $M\satisfies\xpossible\xnecessary\varphi(a)$. That is, $\varphi(a)$ is extension possibly necessary over $M$.
     \item There is a standard finite number $n$ such that for every standard finite $k$ the assertion $\Con(\PAp_k+\neg\Con(\PAp_n+\neg\varphi(a)))$ is true in $M$.
     \item There is a standard finite number $n$ such that for every standard finite $k$ the assertion
      $\Con(\PAp_k+``\PAp_n\proves\varphi(a)")$ is true in $M$.
   \end{enumerate}
\end{theorem}

\begin{proof} The numbers $n$ and $k$ are standard.

($1\to 2$) Assume that $M\satisfies\uppossible\upnecessary\varphi(a)$, where $\varphi\in\Larith$. Since $\uppossible$ implies $\xpossible$, this means $M\satisfies\xpossible\upnecessary\varphi(a)$. By lemma~\ref{Lemma.Necessity-characterization}, since $\varphi\in\Larith$, it follows that $\upnecessary\varphi(a)$ is equivalent to $\xnecessary\varphi(a)$, and so $M\satisfies\xpossible\xnecessary\varphi(a)$, as desired.

($2\to 3$) If $M\satisfies\xpossible\xnecessary\varphi(a)$, then there is an extension $N$ of $M$ with $N\satisfies\xnecessary\varphi(a)$. By lemma~\ref{Lemma.Necessity-characterization} in $N$, there is some standard finite $n$ for which $N\satisfies\neg\Con(\PAp_n+\neg\varphi(a))$. Thus, $M\satisfies\xpossible\neg\Con(\PAp_n+\neg\varphi(a))$, and since this is an arithmetic assertion, it follows by lemma~\ref{Lemma.Possibility-characterization} that $M\satisfies\Con(\PAp_k+\neg\Con(\PAp_n+\neg\varphi(a)))$ for all standard $k$, establishing (3).

($3\to 1$) If there is $n$ for which all those consistency assertions are true in $M$, then by lemma~\ref{Lemma.Possibility-characterization} there is an end-extension in which $\neg\Con(\PAp_n+\neg\varphi(a))$ holds, making $\varphi(a)$ true in all end-extensions of $N$ by lemma~\ref{Lemma.Necessity-characterization}.

($3\iff 4$) Proving a statement is equivalent to inconsistency of the negated statement.
\end{proof}

Theorem~\ref{Theorem.Modalities-are-different} shows that the equivalence of $\uppossible$ and $\xpossible$ does not continue all the way into the full potentialist languages of arithmetic $\Larith^{\uppossible}$ and $\Larith^{\xpossible}$, and Shavrukov and Visser have proved (see remarks after question~\ref{Question.Possibility-equivalence-in-partial-potentialist-language}) that the equivalence can already break down inside ${\uppossible}\Larith$ and ${\xpossible}\Larith$.

\goodbreak
\begin{theorem}[{\cite[theorem 15]{Visser1998:An-overview-of-interpretability-logic}}]\label{Theorem.Standard-cut-definable}
 In any model of arithmetic $M\satisfies\PAp$, the standard cut is definable in the partial potentialist languages ${\uppossible}\Larith$ and ${\xpossible}\Larith$ as follows:
 \begin{align*}
   k\text{ is standard }\quad&\text{if and only if}\quad\upnecessary\Con(\PAp_k),\\
                        &\text{if and only if}\quad\xnecessary\Con(\PAp_k).
 \end{align*}
 Consequently, in any nonstandard model of arithmetic, the induction principle fails in the potentialist languages of arithmetic. Conversely, for formulas in the language of arithmetic $\varphi\in\Larith$, the assertion $\uppossible\varphi(a)$, which is equivalent to $\xpossible\varphi(a)$, is expressible in $\<M,+,\cdot,0,1,\N>$, with a predicate for the standard cut $\N$.
\end{theorem}

\begin{proof}
If $k$ is standard, then $\Con(\PAp_k)$ holds by theorem~\ref{Theorem.Mostowski-reflection-theorem}, and since this is true in all extensions, we may also conclude $\upnecessary\Con(\PAp_k)$ and $\xnecessary\Con(\PAp_k)$. If $k$ is nonstandard, however, then there is always an end-extension where $\Con(\PAp_k)$ fails, for if it hasn't failed yet, then $\PAp_k$ is consistent in the model, and so by the incompleteness theorem we may form a complete consistent Henkin theory extending $\PAp_k+\neg\Con(\PAp_k)$, which provides an end-extension model of $\PAp$ in which $\neg\Con(\PAp_k)$. So $\neg\upnecessary\Con(\PAp_k)$ and hence $\neg\xnecessary\Con(\PAp_k)$ in the original model.

Since the standard cut contains $0$ and is closed under successor, this provides a violation of the induction principle for the partial potentialist language in any nonstandard model.

For the converse direction, if we have a predicate for the standard cut, then $\uppossible\varphi(a)$ and $\xpossible\varphi(a)$ are expressible, since by lemma~\ref{Lemma.Possibility-characterization}, these are both equivalent to asserting $\Con(\PAp_k+\varphi(a))$ for all standard $k$.
\end{proof}

Let me generalize theorem~\ref{Theorem.Standard-cut-definable} to the $\Sigma_n$-elementary extension relations.

\begin{theorem}
 In any model of arithmetic $M\satisfies\PAp$ and for any standard $n$, the standard cut of $M$ is definable in the partial potentialist languages ${\uppossible_n}\Larith$ and ${\xpossible_n}\Larith$ as follows:
 \begin{align*}
   k\text{ is standard }\quad&\text{if and only if}\quad\upnecessary_n\Con(\Tr_n+\PAp_k),\\
                        &\text{if and only if}\quad\xnecessary_n\Con(\Tr_n+\PAp_k).
 \end{align*}
\end{theorem}

\begin{proof}
By $\Tr_n$, I mean the $\Sigma_n$ theory as it is defined by the universal $\Sigma_n$ definition. Because this theory is not c.e.~when $n>1$, the complexity of the consistency assertion $\Con(\Tr_n+\PAp_k)$ will generally rise to $\Pi_{n+1}$. Note that reference is made to $\Tr_n$ by its definition, and so this expression is in effect re-interpreted in the extensions of $M$ for the purpose of evaluating the modal assertion in $M$.

If $k$ is standard, then we get $\Con(\Tr_n+\PAp_k)$ in $M$ and its extensions by the Mostowski reflection theorem (theorem~\ref{Theorem.Mostowski-reflection-theorem}). Conversely, suppose that $k$ is nonstandard in $M$. It suffices to find a $\Sigma_n$-elementary end-extension of $M$ in which the theory $\Tr_n+\PAp_k$ is inconsistent. We may assume $\Tr_n+\PAp_k$ is consistent in $M$. Let $M^+$ be any elementary end-extension of $M$. In $M^+$, let $\tau$ be the conjunction of a (nonstandard) finite part of $\Tr_n^{M^+}$, long enough so that every assertion of $\Tr_n^M$ appears as a conjunct of $\tau$. By elementarity, the theory $\tau+\PAp_k$ is consistent in $M^+$. By the incompleteness theorem in $M^+$ applied to this theory, it follows that $\tau+\PAp_k+\neg\Con(\tau+\PAp_k)$ is consistent in $M^+$. The Henkin model of this theory as constructed in $M^+$ provides an end-extension $N\satisfies\tau+\PAp_k+\neg\Con(\tau+\PAp_k)$. Since $N\satisfies\tau$, it follows that $M\elesub_{\Sigma_n} N$. Since $k$ is nonstandard, it follows that $N\satisfies\PAp$. Since $\tau$ is part of $\Tr_n^{M^+}$ and true in $N$, according to $M^+$, it follows that the conjuncts of $\tau$ are part of $\Tr_n^N$, and thus $N\satisfies\neg\Con(\Tr_n+\PAp_k)$, as desired.
\end{proof}

Although lemma~\ref{Lemma.Possibility-characterization} shows that $\uppossible\varphi(a)$ is equivalent to $\xpossible\varphi(a)$ when $\varphi$ is an assertion in the language of arithmetic, let me now prove that this does not extend to formulas $\varphi$ in the full potentialist language of arithmetic.

\goodbreak
\begin{theorem}\label{Theorem.Modalities-are-different}
 There is a model of arithmetic $M\satisfies\PAp$ and a sentence $\sigma\in\Larith^{\uppossible}$ in the potentialist language of arithmetic, such that $M\satisfies\xpossible\sigma\wedge\neg\uppossible\sigma$.
\end{theorem}

\begin{proof} The sentence $\sigma$ will assert that the (standard) halting problem is in the standard system of the model, that is, the collection of subsets of $\N$ coded in the model. This is expressible in either of the languages $\Larith^{\uppossible}$ or $\Larith^{\xpossible}$, because in these potentialist languages, we can refer to the standard cut. Specifically, take $\sigma$ to be the assertion $\exists h\forall e{\in}\N\, (e\in h\iff\exists s{\in}\N\, \varphi_{e,s}(e)\converges)$, which can be translated in the potentialist language as
 $$\sigma\quad=\quad\exists h\forall e\left[\necessary\Con(\PAp_e)\to(e\in h\iff \exists s\, \varphi_{e,s}(e)\converges\wedge\necessary\Con(\PAp_s))\right].$$
Since the modality $\necessary$ here is applied only to arithmetic assertions, it doesn't matter whether we use $\upnecessary$ or $\xnecessary$, since they are equivalent for this case.

Since $\PAp$ is a consistent c.e. theory, it has a model $M$ whose elementary diagram is low and therefore does not have the halting problem $0'$ in $\SSy(M)$. Therefore also every end-extension of $M$ will fail to have $0'$ in its standard system, and so $M\not\satisfies\uppossible\sigma$. Yet, a simple compactness argument finds an elementary extension $M\elesub N$ with $0'\in\SSy(N)$, and so $M\satisfies\xpossible\sigma$.
\end{proof}

So the two modalities $\uppossible$ and $\xpossible$ are not always identical when applied to assertions in the full potentialist language of arithmetic. The sentence $\sigma$ used in the proof of theorem~\ref{Theorem.Modalities-are-different} is in the potentialist language $\Larith^{\possible}$, but not in the partial language ${\possible}\Larith$, because it has modal operators appearing under the scope of number quantifiers. Therefore it is natural to inquire exactly where the difference between $\uppossible$ and $\xpossible$ first arises, and in particular, whether their equivalence in the language of arithmetic $\Larith$ extends into the partial potentialist languages ${\uppossible}\Larith$ and ${\xpossible}\Larith$.

\begin{question}\label{Question.Possibility-equivalence-in-partial-potentialist-language}
 Are $\uppossible\varphi(a)$ and $\xpossible\varphi(a)$ equivalent for $\varphi$ in the partial potentialist languages ${\uppossible}\Larith$ and ${\xpossible}\Larith$?
\end{question}

This question has evidently been answered negatively by Visser and Shavrukov in unpublished work, some of which may have appeared in an 2016 email exchange between them, in which Visser proved the existence of separating sentences as an abstract consequence of~\cite{Visser2015:Semantic-completion-of-PA-is-complete}, with Shavrukov proving specifically that there is an arithmetic sentence $\varphi$ and a model of arithmetic $M$ satisfying both $\xpossible\xnecessary\xpossible\neg\varphi$ and $\upnecessary\uppossible\upnecessary\varphi$.

Similar questions arise with essentially all of the modal operators appearing in this article, such as $\consuppossible$, $\consxpossible$, $\soliduppossible$, $\solidxpossible$ and all the others. Exactly how much do all these modal operators agree? When they do disagree, at what level of complexity of assertions do the disagreements begin to show up?

Let me generalize the idea behind the proof of theorem~\ref{Theorem.Modalities-are-different} to show that actually arithmetic truth for the standard model is expressible in the language of extensional-potentialist arithmetic. Indeed, even projective truth is expressible.

\begin{theorem}
 The standard truth predicate for the standard model $\N$ is definable by a formula in the extension-potentialist language of arithmetic $\Larith^{\xpossible}$.
\end{theorem}

\begin{proof} The standard truth predicate is unique and we can always add it to the standard system of a suitable extension. So we can define it in the modal language by saying $\Tr(x)$ holds if and only if $x$ is standard and it is possible that there is a number $t$ coding the standard truth predicate in the standard system, and $x\in t$. To say that a number $t$ codes the standard truth predicate in the standard system is expressible by a single property about $t$ in the potentialist language, since the standard cut is definable in the potentialist language, and so we can say that the set coded by $t$ obeys the Tarskian recursion for truth in the standard model.
\end{proof}

Let me push the idea much harder. In fact, projective truth, which is to say, second-order truth over the standard model of arithmetic $\N$, is expressible in the potentialist language of arithmetic in any model of arithmetic.

\goodbreak
\begin{theorem}\label{Theorem.Projective-truth-is-expressible}
 Projective truth (second-order arithmetic truth) is expressible in the potentialist language of arithmetic $\Larith^{\xpossible}$. That is, for every projective formula $\varphi(x)$, there is a formula $\varphi^*(x)\in\Larith^{\xpossible}$ such that for any model of arithmetic $M\satisfies\PAp$ any real $a\in\SSy(M)$, coded by $\bar a\in M$, the projective assertion $\varphi(a)$ holds if and only if $M\satisfies\varphi^*(\bar a)$.
\end{theorem}

In particular, the $\xpossible$-potentialist language of arithmetic can express whether a given binary relation is well-founded, whether there is a transitive model of \ZFC, whether $0^\sharp$ exists, whether $x^\sharp$ exists for every real $x$, whether $\omega_1$ is inaccessible in $L$, whether there is a transitive model of \ZFC\ with a proper class of Woodin cardinals, and more. The full potentialist language of arithmetic for the $\xpossible$ modality, therefore, is a very strong language. As a philosophical consequence, this refutes what might otherwise have been a natural initial expectation, namely, that arithmetic potentialists generally have only a weak capacity to express actual arithmetic truth. In fact, they can fully and accurately express second-order arithmetic truth.

\begin{proof}
The idea is to replace quantification over the reals $\exists x$ with potentialist number quantification $\xpossible\exists h$, regarding the number $h$ as coding a real in the standard system, so that we replace further references to $k\in x$ with $k\in h$, for standard $k$, and replace all number quantifiers with quantifiers over the standard part of the model, as in the previous theorem. The point is that by compactness any real can be added to the standard system of a model by moving to a suitable extension, even an elementary extension, and so quantifying over reals is the same as quantifying over the possible reals that could be coded in the standard system of the model.
\end{proof}

The result may be connected with the main result of~\cite{Visser2015:Semantic-completion-of-PA-is-complete}, which shows that the collection of possibly necessary sentences over \PA\ is $\Pi^1_1$-complete.

\begin{question}
 How expressive is the end-extensional potentialist language $\Larith^{\uppossible}$? Can it express arithmetic truth? Can it express projective truth? Is the halting problem definable in every model of arithmetic by a $\Larith^{\uppossible}$ formula?
\end{question}

I am inclined to think not, and the problem is how to show that $\Larith^{\uppossible}$ is weaker than $\Larith^{\xpossible}$.

\begin{question}\label{Question.Up-expressible-from-x?}
 Is $\uppossible$ expressible in the potentialist language $\Larith^{\xpossible}$?
\end{question}

In other words, in the extension-potentialist language $\Larith^{\xpossible}$, can we define a formula equivalent to $\uppossible\varphi(x)$?

\begin{question}
Is $\xpossible$ expressible from $\uppossible$?
\end{question}

\enlargethispage{20pt}%
Perhaps $\uppossible$ cannot define truth in the standard model, for abstract reasons. For example, perhaps every $\Larith^{\uppossible}$-definable set is computable from the oracle $0^{(\omega)}$, which would mean that projective truth would not be $\Larith^{\uppossible}$ expressible, leading to the conclusion that $\xpossible$ is not expressible from $\uppossible$. I am not sure.\goodbreak

Similar questions arise with essentially all the other modalities considered in this article. There is plenty of work here to occupy all of us for the future.

\section{The universal algorithm}\label{Section.Universal-algorithm}

The universal algorithm is an absolutely beautiful result due to W. Hugh Woodin, upon which much of the modal analysis of this article is based. The main fact is that there is a Turing machine program that can in principle enumerate any desired finite sequence of numbers, if only it is run in a suitable universe; and furthermore, in any model of arithmetic, one can realize any desired further extension of the enumerated sequence by moving to a taller model of arithmetic end-extending the previous one.

It will be convenient to use an enumeration model of Turing computability, by which we view a Turing machine program as providing a means to computably enumerate a list of numbers. We start a program running, and it generates a list of numbers, possibly empty, possibly finite, possibly infinite, possibly with repetition.

The history of the universal algorithm result (theorem \ref{Theorem.Universal-algorithm}) involves several instances of independent rediscovery of closely related or essentially similar results, alternative arguments and generalizations. Woodin's original theorem, for countable models only, appears in~\cite{Woodin2011:A-potential-subtlety-concerning-the-distinction-between-determinism-and-nondeterminism}, with further results and discussion by Blanck and Enayat~\cite{BlanckEnayat2017:Marginalia-on-a-theorem-of-Woodin, Blanck2017:Dissertation:Contributions-to-the-metamathematics-of-arithmetic}, including notably their extension of the result from countable to arbitrary models, as well as in a series of my blog posts~\cite{Hamkins.blog2016:Every-function-can-be-computable, Hamkins.blog2017:A-program-that-accepts-exactly-any-desired-set-in-the-right-universe, Hamkins.blog2017:The-universal-algorithm-a-new-simple-proof-of-Woodins-theorem}, the last of which provides a simplified proof of the theorem, which I provide below (and an earlier version of this article was posted on the Arxiv at \cite{Hamkins:The-modal-logic-of-arithmetic-potentialism}). It turns out that Volodya Shavrukov had advanced an essentially similar argument in an email message (February 16, 2012) to Ali Enayat and Albert Visser, under the slogan, ``On risks of accruing assets against increasingly better advice.'' Shavrukov has pointed out that this argument is closely related to and can be seen as an instance of the construction of Berarducci in~\cite{Berarducci1990:The-interpretability-logic-of-PA}, with Woodin's construction similarly seen (after the fact) as a certain instance of the general Berarducci-Japaridze~\cite{Japaridze1994:A-simple-proof-of-arithmetical-completeness-for-Pi11-conservativity-logic}
construction, a relation of part to whole, he says, ``a curious instance of the part growing more glamorous than the whole,'' in light of the broad appeal of Woodin's theorem.  Albert Visser has pointed out a similar affinity between the universal algorithm and the classical proof-theoretic `exile' argument (for example, see `refugee' in~\cite{ArtemovBeklemishev2004:Provability-logic}), for the universal algorithm is allowed to succeed at stage $n$ exactly when it finds a particular kind of proof that this will not be the last successful stage, just as the exile is allowed to enter a country only when he can also prove that he will eventually move on. Weaker incipient forms of the result are due to Mostowski~\cite{Mostowski1960:A-generalization-of-the-incompleteness-theorem} and Kripke~\cite{Kripke1962:Flexible-predicates-of-formal-number-theory}. Meanwhile, in current joint work, Woodin and I have provided a set-theoretic analogue of the result in~\cite{HamkinsWoodin:The-universal-finite-set}, and in joint work with Kameryn Williams~\cite{HamkinsWilliams2021:The-universal-finite-sequence}, we have a $\Sigma_1$-definable analogue for end-extensional set-theoretic potentialism, including models of $\ZFC^-$, and now also the $\Sigma_n$ analogues in \cite{Hamkins2024:Every-countable-model-of-arithmetic-or-set-theory-has-a-pointwise-definable-end-extension}.

\begin{theorem}\label{Theorem.Universal-algorithm} Suppose that $\PAp$ is a consistent c.e.~theory extending \PA. Then there is a Turing machine program $e$ with the following properties.
  \begin{enumerate}
    \item $\PA$ proves that the sequence enumerated by $e$ is finite.
    \item For any model $M$ of $\PAp$ in which $e$ enumerates a finite sequence $s$, possibly nonstandard, and any $t\in M$ extending $s$, there is an end-extension of $M$ to a model $N\satisfies\PAp$ in which $e$ enumerates exactly $t$.
    \item In the standard model $\N$ the program enumerates the empty sequence.
    \item Consequently, for every finite sequence $s$, there is a model of $\PAp$ in which the program enumerates exactly $s$.
  \end{enumerate}
\end{theorem}

\begin{proof} This is my simplified argument. I shall describe an algorithm $e$ that enumerates a sequence of numbers, releasing them in batches at successful stages. Stage $n$ is successful, if all earlier stages are successful and there is a proof in a theory fragment $\PAp_{k_n}$, with $k_n$ strictly smaller than all earlier $k_i$ for $i<n$, of a statement of the form, ``it is not the case that program $e$ enumerates the elements of the sequence $s$ at stage $n$ as its last successful stage'' where $s$ is some explicitly listed sequence of numbers. In this case, for the first such proof that is found, the algorithm $e$ adds the numbers appearing in $s$ to the sequence it is enumerating, and then continues, trying for success in the next stage. The existence of such a program $e$, which iteratively searches for such proofs about itself, follows from the Kleene recursion theorem.

Since the numbers $k_n$ are descending as the algorithm proceeds, there can be only finitely many successful stages and so the program will enumerate altogether a finite sequence. So statement (1) holds.

If stage $n$ is successful, it is because of a certain proof found using axioms from the theory $\PAp_{k_n}$. I claim that $k_n$ must be nonstandard. The reason is that if $n$ is the last successful stage and $k$ is any standard number, then by the Mostowski reflection theorem (theorem~\ref{Theorem.Mostowski-reflection-theorem}), the theory $\Tr_k$ is consistent, and for $k\geq 2$ this would include the assertion that $e$ enumerated exactly the sequence that it did enumerate in $M$ with exactly those stages being successful. So there can be no proof from $\PAp_k$ in $M$ that this was not the case. So every $k_n$ is nonstandard.

It follows that there can be no successful stage in the standard model, so the program $e$ enumerates the empty sequence in the standard model $\N$, fulfilling statement (3). Statement (4) follows from this, since if there were some finite sequence $s$ unrealized in models of $\PAp$, then this would be provable and so there would be a successful stage in the standard model, contrary to what we just observed.

For the extension property of statement (2), suppose that $M\satisfies\PAp$ and program $e$ enumerates the (possibly nonstandard) finite sequence $s$ in $M$ and $t$ is an arbitrary finite sequence in $M$ extending $s$. Let $n$ be the first unsuccessful stage of $e$ in $M$. I may assume $M$ is nonstandard by moving to an elementary end-extension if necessary. Let $k$ be any nonstandard number in $M$ that is smaller than all $k_i$ for $i<n$. Since stage $n$ was not successful, there must not be any proof in $M$ from $\PAp_k$ refuting the assertion that $n$ is the last successful stage and enumerates exactly the rest of the elements of $t$ beyond $s$ at stage $n$. In other words, $M$ must think that it is consistent with $\PAp_k$ that $e$ does have exactly $n$ successful stages and enumerates exactly the rest of $t$ at stage $n$. If $T$ is the theory asserting this, then we may build the Henkin model of $T$ inside $M$, and this model will provide an end-extension of $M$ to a model $N$ that satisfies $\PAp_k$ in which $e$ enumerates exactly the sequence $t$ altogether. Since $k$ is nonstandard, the theory $\PAp_k$ includes the entire standard theory $\PAp$, and so I have found the desired extension to fulfill statement (2).
\end{proof}

Let me now describe an alternative related algorithm, what I call the \emph{one-at-a-time} universal algorithm. Namely, program $\hat e$ will enumerate a finite sequence of numbers $a_0$, $a_1$ and so on, adding just one number at each successful stage, but only finitely many numbers will be enumerated. The number $a_n$ is defined, if the earlier numbers have already been enumerated and there is a proof from $\PAp_{k_n}$, with $k_n$ strictly smaller than $k_i$ for $i<n$, of a statement of the form, ``program $\hat e$ does not enumerate number $a$ at stage $n$ as $a_n$ as its next and last number.'' Since the $k_n$ are descending, there will be only finitely many numbers enumerated. If $\hat e$ enumerates $s$ in $M\satisfies\PAp$, then by the Henkin theory argument, we can add any desired next number as $a_n$ in some end-extension $N$ of $M$, with only that number added, where $n$ was the first undefined number in $M$. In other words, for the one-at-a-time universal algorithm, we weaken the extension property of statement (2) to achieve only that the sequence can be extended to add any desired next number (only one), rather than an arbitrary sequence of numbers. Indeed, the one-at-a-time algorithm $\hat e$ will not have the full extension property, since once $a_0$ is defined, then the total length of the enumerated sequence will be bounded by $k_0$, so we cannot extend to arbitrary nonstandard lengths.

Nevertheless, the one-at-a-time universal algorithm is fully general, in the sense that we can interpret the individual numbers on its sequence each as a finite sequence of numbers, to be concatenated. That is, from the enumerated sequence $a_0,a_1,\ldots,a_n$ we derive a corresponding concatenated sequence $s_0\concat s_1\concat\cdots\concat s_n$, where $s_k$ is the finite sequence coded by $a_k$. One thereby easily achieves the full extension property for this derived sequence, since if you can always add any desired individual next number $a_{n+1}$ to the sequence enumerated by $\hat e$, then you can always add any desired (possibly nonstandard) finite sequence $s_{n+1}$ to the derived concatenated sequence. Indeed, this is exactly how the universal algorithm is derived from the one-at-a-time universal algorithm. In this way, one sees that the full extension property reduces to the one-at-a-time extension property and indeed the two notions are essentially equivalent, in the sense that we can transform any one-at-a-time universal sequence into a fully universal sequence. Meanwhile, in certain circumstances the one-at-a-time sequence is simpler to consider.

Statement (3) of theorem~\ref{Theorem.Universal-algorithm} generalizes to the following.

\begin{theorem}
 In any model $M\satisfies\PAp$ that thinks $\PAp$ is $\Sigma_1$-sound, the universal algorithm $e$ of the proof of theorem~\ref{Theorem.Universal-algorithm} enumerates the empty sequence.
\end{theorem}

\begin{proof}
 If $e$ has a successful stage in $M$, then it has a last successful stage $n$, and this stage is successful because of a certain proof in $M$ that $e$ does not have exactly that many successful stages, enumerating exactly some explicitly listed sequence $s$ at stage $n$. But of course, $M$ does have proofs that the stages up to and including $n$ are successful and that it enumerates exactly the sequences that it does enumerate at those stages, since such proofs can be constructed from the actual computation itself in $M$. It follows that in $M$, there is a proof that stage $n+1$ will be successful, even though it won't actually be successful in $M$. This is a $\Sigma_1$ assertion that is provable from $\PAp$ in $M$ but not true in $M$, and so the theory is not held to be $\Sigma_1$-sound in $M$.
\end{proof}

Of course, $\Sigma_1$-soundness is a strengthening of $\Con(\PAp)$, and in any model of $\neg\Con(\PAp)$ the algorithm will find proofs and therefore have a successful stage. Blanck and Enayat~\cite{BlanckEnayat2017:Marginalia-on-a-theorem-of-Woodin} prove that their version of the universal algorithm, as well as Woodin's original algorithm, has the property that it enumerates a nonempty sequence if and only if $\neg\Con(\PAp)$. It seems that the argument can be made to work also for the algorithm here.

I find it interesting to consider theories such as the following. By the \Godel-Carnap fixed point lemma, there is a sentence $\sigma$ asserting ``the universal algorithm $e$, relative to the theory $\PAp=\PA+\sigma$, has at least one successful stage.'' If this theory were inconsistent, then in the standard model the theory $\PAp$ would prove everything, and this would cause a successful stage in the corresponding universal algorithm, which would make the sentence $\sigma$ itself and hence also $\PAp$ true in the standard model, contrary to assumption. So the theory is consistent. But it cannot actually be true in the standard model, since in that case the algorithm $e$ would have to have no successful stages in $\N$ by statement (3) of theorem~\ref{Theorem.Universal-algorithm}, contrary to the presumed truth of $\sigma$ in $\N$. Note that this theory has no models in which the algorithm is never successful, because that is precisely what $\sigma$ asserts not to occur.

\goodbreak
\begin{corollary}\label{Corollary.Universal-infinite-sequence}
Let $e$ be the universal algorithm of theorem~\ref{Theorem.Universal-algorithm}. Then\nobreak
    \begin{enumerate}
      \item For any infinite sequence of natural numbers $\<a_0,a_1,a_2,\ldots>$, there is a model $M$ of $\PAp$ in which program $e$ enumerates a nonstandard finite sequence extending it.
      \item If $M$ is any model of $\PAp$ in which program $e$ enumerates some (possibly nonstandard) finite sequence $s$, and $S$ is any $M$-definable infinite sequence extending $s$, then there is an end-extension of $M$ satisfying $\PAp$ in which $e$ enumerates a sequence starting with $S$.
    \end{enumerate}
\end{corollary}

\begin{proof}
For statement (1), fix the sequence $\<a_0,a_1,a_2,\ldots>$. By a simple compactness argument, there is a model $M$ of $\PAp$ in which program $e$ enumerates exactly $a_n$ as its $n^{th}$ element.

For statement (2), if $e$ enumerates $s$ in $M$, a model of $\PAp$, and $S$ is an $M$-infinite sequence definable in $M$ and extending $s$, then apply such a compactness argument inside $M$ using a nonstandard fragment $\PAp_k$. (If $M$ is the standard model, apply statement (1) instead.)
\end{proof}

Corollary~\ref{Corollary.Universal-infinite-sequence} shows the sense in which every function on the natural numbers can become computable, and indeed, computed always by the same universal program, if only the program is run inside the right model of arithmetic.

\begin{observation}\label{Observation.Necessary-but-not-Sigma1}
 There is a statement $\psi$ that is independent of $\PA$ and not provably equivalent to any $\Sigma_1$ assertion, such that $\psi$ can become necessarily true in all end-extensions of a model of $\PA$.
\end{observation}

\begin{proof}
 Let $\psi$ be the assertion, ``the sequence enumerated by the universal algorithm $e$ for the theory $\PA$ does not begin with the number $17$.'' This statement is independent of $\PA$, since theorem~\ref{Theorem.Universal-algorithm} shows that there are models in which it is true and models in which it is false. It is a $\Pi_1$ assertion, since the negation asserts that the sequence does begin with $17$, which is a $\Sigma_1$ assertion; and it cannot be provably equivalent to any $\Sigma_1$ assertion, since it is true in any model where the universal sequence is empty, but false in an end-extension where the universal algorithm sequence does begin with $17$. In any model $M$ where the universal sequence is nonempty but does not begin with $17$, however, then $\psi$ is necessary, since in no end-extension will it begin with $17$ if it didn't already. So there are models with $M\satisfies\upnecessary\psi$ even though $\psi$ is not provably $\Sigma_1$.
\end{proof}

In confirmation of my remarks after lemma~\ref{Lemma.Necessity-characterization}, however, note that in the model $M$ where $\upnecessary\psi$ holds, there is a true $\Sigma_1$ statement that provably implies $\psi$, namely, the assertion that the universal algorithm is not empty and begins with something other than $17$.

\subsection{Generalization to $\Sigma_n$-elementary end-extensions}

I would like now to generalize the universal algorithm argument to the case of $\Sigma_n$-elementary end-extensions. I shall use this theorem when analyzing the potentialist validities of $\uppossible_n$ and $\xpossible_n$ in section~\ref{Section.Other-forms-of-arithmetic-potentialism}.

\goodbreak
\begin{theorem}\label{Theorem.Universal-algorithm-Sigma-n}
Suppose that $\PAp$ is a consistent c.e.~theory extending \PA, and $n$ is any standard-finite natural number. Then there is an oracle Turing machine program $\tilde e$ with the following properties.
  \begin{enumerate}
    \item $\PA$ proves that the sequence enumerated by $\tilde e$, with any oracle, is finite.
    \item For any model $M$ of $\PAp$ in which $\tilde e$, using oracle $0^{(n)}$ of $M$, enumerates a finite sequence $s$, possibly nonstandard, and any $t\in M$ extending $s$, there is a $\Sigma_n$-elementary end-extension of $M$ to a model $N\satisfies\PAp$ in which $\tilde e$, with oracle $0^{(n)}$ of $N$, enumerates exactly $t$.
    \item If $\N\satisfies\PAp$, then in the standard model $\N$ with the actual $0^{(n)}$, the program enumerates the empty sequence.
    \item Consequently, if $\N\satisfies\PAp$, then for every finite sequence $s$, there is a model of $\PAp$ in which the program, using oracle $0^{(n)}$ of the model, enumerates exactly $s$.
  \end{enumerate}
\end{theorem}

\begin{proof} I modify the proof of theorem~\ref{Theorem.Universal-algorithm} by searching for proofs not in $\PAp$ but in the theory $\Tr_n+\PAp$, where $\Tr_n$ is the theory of $\Sigma_n$ truth, which is computable from the oracle $0^{(n)}$. This theory is of course no longer computably enumerable, but it is uniformly computable for standard $n$ from the oracle $0^{(n)}$ in any model of arithmetic.

So the universal program $\tilde e$ here searches for a proof in $\Tr_n+\PAp_k$, using the oracle $0^{(n)}$ to get access to $\Tr_n$ and insisting on a strictly smaller fragment $k$ each time (while $n$ is fixed), that program $\tilde e$, when equipped with oracle $0^{(n)}$, does not enumerate a certain specifically listed sequence of numbers, and when found, it enumerates that sequence anyway.

The proof of theorem~\ref{Theorem.Universal-algorithm} adapts easily to this new context. Since the theory fragment $\PAp_k$ is descending, there will be only finitely many successful stages and so the enumerated sequence will be finite. For the extension property, suppose that the program $\tilde e$ enumerates a sequence $s$ in a model $M\satisfies\PAp$, with stage $m$ the first unsuccessful stage, and that $t$ is a finite extension of $s$ in $M$. All the numbers $k_i$ for $i<m$, if any, are nonstandard by the same argument as before, appealing to the Mostowski reflection theorem. Let $k$ be any nonstandard number smaller than $k_i$ for $i<m$. Since stage $m$ was not successful, there is no proof in $M$ from $\Tr_n+\PAp_k$ that stage $m$ is the last successful stage, enumerating the rest of $t$ beyond $s$. Thus, this theory is consistent in $M$, and so by considering the Henkin theory, we get a model $N$ of $\PAp_k$ in which $\tilde e$, using oracle $0^{(n)}$ of the new model, enumerates exactly $t$. Note that because $N$ satisfies the theory $\Tr_n$ as defined in $M$, it follows that this is a $\Sigma_n$-elementary end-extension $M\elesub_{\Sigma_n} N$, and so the oracle $0^{(n)}$ as defined in $N$ agrees with the oracle as defined in $M$. So the earlier part of the computation of $\tilde e$ is the same in $M$ as in $N$. And since $k$ is nonstandard, $N$ is a model of $\PAp$. So we've found the desired extension.

So far, this gives a separate $\tilde e$ for each $n$, but actually, $\tilde e$ can check which $n$ it is using by consulting the oracle, because there is a program which on input $0$ with oracle $0^{(n)}$ outputs $n$. So we can have a single program $\tilde e$ that works uniformly with all $n$.
\end{proof}

Theorem~\ref{Theorem.Universal-algorithm-Sigma-n} was observed independently by Rasmus Blanck~\cite{Blanck:Hierarchy-incompleteness-results-for-arithmetically-definable-fragments-of-arithmetic}.

\goodbreak
\subsection{Applications of the universal algorithm}

I should like briefly to explain a few applications of the universal algorithm. It turns out that several interesting classical results concerning the models of arithmetic can be deduced as immediate consequences of the universal algorithm result.

For example, the universal algorithm easily provides an infinite list of mutually independent $\Pi^0_1$-assertions. The existence of such a family was first proved by Mostowski~\cite{Mostowski1960:A-generalization-of-the-incompleteness-theorem} and independently by Kripke~\cite{Kripke1962:Flexible-predicates-of-formal-number-theory}. Some logicians have emphasized that the universal algorithm provided in corollary~\ref{Corollary.Universal-infinite-sequence} should be seen as a consequence of the independent $\Pi^0_1$-sentences.\footnote{For example, see the comments on~\cite{Hamkins.blog2016:Every-function-can-be-computable}.} To my way of thinking, however, we should view the implication more naturally in the other direction, in light of the easy argument from the universal algorithm to the independent $\Pi^0_1$ sentences:

\begin{theorem}
  There are infinitely many mutually independent $\Pi^0_1$ sentences $\eta_0$, $\eta_1$, $\eta_2$, and so on. Any desired true/false pattern for these sentences is consistent with $\PA$.
\end{theorem}

\begin{proof}
 Let $\eta_k$ be the assertion that $k$ does not appear on the universal sequence, meaning the sequence enumerated by the universal algorithm $e$ of theorem~\ref{Theorem.Universal-algorithm} using the theory \PA. Since the universal sequence is empty in $\N$, these statements are all true in the standard model. Since the extension property of the universal sequence allows us to find an end-extension $N$ of any given model $M$ that adds precisely any desired finitely many additional numbers to the sequence, we can make the $\eta_k$ become true or false in any finite desired combination, and so by compactness we can also achieve any infinite pattern consistently as well.
\end{proof}

The negations of these sentences form a natural independent family of buttons. A \emph{button} in a Kripke model is a statement $\rho$ for which $\possible\necessary\rho$ holds in every world. The button is \emph{pushed} in a world where $\necessary\rho$ holds, and otherwise \emph{unpushed}. A family of buttons is \emph{independent}, if in any model, any of the buttons can be pushed in a suitable extension without pushing the others.

\begin{theorem}
 There are infinitely many $\Sigma^0_1$ sentences $\rho_0$, $\rho_1$, $\rho_2$, and so on, which form an independent family of buttons for the models of \PA\ under end-extension, unpushed in the standard model. Indeed, for any $M\satisfies\PA$ and any $I\of\N$ coded in $M$, there is an end-extension $N$ of $M$ such that $N\satisfies\rho_k$ for all $k\in I$, and otherwise the truth values of $\rho_i$ are unchanged from $M$ to $N$ for $i\notin I$.
\end{theorem}

\begin{proof}
Let $\rho_k$ be the assertion that $k$ appears on the universal finite sequence. In the standard model, no numbers appear on that sequence, and in a suitable end-extension, any number or coded set of numbers can be added without adding any others.
\end{proof}

Next, consider the case of Orey sentences~\cite{Orey1961:Relative-interpretations}. An \emph{Orey sentence} is a sentence $\sigma$ in the language of arithmetic, such that every model of \PA\ has end-extensions where $\sigma$ is true and others where $\sigma$ is false.
In other words, an Orey sentence is a \emph{switch} in the Kripke model consisting of the models of \PA\ under end-extension, a sentence $\sigma$ for which $\uppossible\sigma$ and $\uppossible\neg\sigma$ are true in every model of \PA. A switch in a Kripke model is a statement that can be turned on and off from any world by moving to a suitable accessible world.

The universal algorithm easily provides numerous Orey sentences and indeed infinitely many independent Orey sentences. See also~\cite[corollary~7.11]{Blanck2017:Dissertation:Contributions-to-the-metamathematics-of-arithmetic}.

\goodbreak
\begin{theorem}\label{Theorem.Infinitely-many-independent-Orey-sentences}
 There is an infinite list of mutually independent Orey sentences, or switches, in the language of arithmetic. That is, there are sentences $\sigma_0$, $\sigma_1$, $\sigma_2$, and so on, such that for any model of arithmetic $M\satisfies\PAp$ and any desired pattern of truth for any finitely many of those sentences, there is an end-extension $N\satisfies\PAp$ in which that pattern of truth is realized. Indeed, for any $I\of\N$ in the standard system of $M$, there is an end-extension $N$ of $M$ such that $N\satisfies\sigma_k$ for exactly $k\in I$.
\end{theorem}

\begin{proof}
Let $\sigma_k$ assert simply that the $k^{th}$ binary digit of the last number on the sequence enumerated by the universal algorithm is $1$. This is a $\Sigma_2$ assertion, with the main cause of complexity being that one must assert that one has the \emph{last} number on the universal sequence. Since theorem~\ref{Theorem.Universal-algorithm} allows us to arrange in an end-extension that any desired number appears as the last number on the sequence, we can therefore arrange any particular desired pattern for finitely many of its binary digits. Similarly, for any set $I\of\N$ coded in the model, we can arrange in some end-extension that the binary digits of the last number on the universal sequence conform with $I$, thereby ensuring that $\sigma_k$ holds for exactly the $k$ in $I$, as desired.
\end{proof}

We may similarly use the universal sequence to derive the ``flexible'' formula result of Kripke~\cite{Kripke1962:Flexible-predicates-of-formal-number-theory}. This was also observed by Rasmus Blanck~\cite{Blanck:Hierarchy-incompleteness-results-for-arithmetically-definable-fragments-of-arithmetic}. Here, I also achieve a uniform version of the result in statement (2).

\begin{theorem}\label{Theorem.Flexible-formula}\
 \begin{enumerate}
   \item For every $n\geq 2$, there is a $\Sigma_n$ formula $\sigma(x)$ such that for any model of arithmetic $M$ and any $\Sigma_n$ formula $\phi(x)$, there is an end-extension $N$ of $M$ such that $N\satisfies\forall x\, \bigl(\sigma(x)\iff\phi(x)\bigr)$.
   \item Indeed, there is a computable sequence of formulas $\sigma_n(x)$ for $n\geq 2$, with $\sigma_n$ having complexity $\Sigma_n$, such that for any model of arithmetic $M$ and any sequence of formulas $\phi_n$ coded in $M$, with $\phi_n$ of complexity $\Sigma_n$, there is an end-extension $N$ of $M$ with $N\satisfies\forall x\, \bigl(\sigma_n(x)\iff\phi_n(x)\bigr)$ for all $n\geq 2$.
 \end{enumerate}
\end{theorem}

\begin{proof}
For statement (1), let $\sigma(x)=\Phi_n(k,x)$, where $k$ is the last element of the universal finite sequence and $\Phi_n(k,x)$ is a fixed universal $\Sigma_n$ formula, meaning that every $\Sigma_n$ formula arises as the $k$th section of it for some $k$. The assertion $\sigma(x)$ has complexity $\Sigma_n$, if $n\geq 2$, since $\sigma(x)\iff \exists k\, \Phi_n(k,x)\wedge ``k$ is the final element of the universal sequence.'' Since we can arrange that the final element of the universal sequence is any desired $k$ in a suitable end-extension $N$ of $M$, we can make $\sigma(x)$ agree with any desired $\Sigma_n$ formula.

For the uniform result, simply let $\sigma_n(x)=\Phi_n(k_n,x)$, where $k_n$ is the $n^{th}$ element of the finite sequence added at the last successful stage of the universal sequence. Since we can find an end-extension where the final stage of the universal sequence adds any desired coded sequence, we can arrange that it picks out the indices of the formulas $\phi_n$, thereby obtaining the uniform flexibility property as desired.
\end{proof}

Similar arguments extend these results to the case of $\Sigma_n$-elementary extensions. See also~\cite[theorem~7.21]{Blanck2017:Dissertation:Contributions-to-the-metamathematics-of-arithmetic}.

\goodbreak\begin{theorem}\label{Theorem.Independent-switches-Sigma_n}\
 \begin{enumerate}
   \item For any $n$, there is an infinite list of independent buttons for $\Sigma_n$-elementary end-extensions of models of arithmetic. That is, there are statements $\rho_0$, $\rho_1$, $\rho_2$, and so on, of complexity $\Sigma_{n+1}$ in the language of arithmetic, such that for any $M\satisfies\PA$ and any $I\of \N$ coded in $M$, there is an end-extension $N$ of $M$ with $N\satisfies\rho_k$ for all $k\in I$ and otherwise the truth values of $\rho_i$ are unchanged from $M$ to $N$.
   \item For any $n$, there is an infinite list of independent switches for $\Sigma_n$-elementary end-extensions of models of arithmetic. That is, there are sentences $\sigma_0$, $\sigma_1$, $\sigma_2$, and so on, of complexity $\Sigma_{n+2}$ in the language of arithmetic, such that for any model of arithmetic $M\satisfies\PAp$ and any desired pattern of truth $I\of\N$ in the standard system of $M$, there is a $\Sigma_n$-elementary end-extension $M\elesub_{\Sigma_n}N\satisfies\PAp$ in which that pattern of truth is realized: $N\satisfies\sigma_k$ precisely for $k\in I$.
   \item For any $n$ and any $m\geq n+2$, there is a `flexible' $\Sigma_m$ formula $\sigma(x)$, one for which for every model of arithmetic $M\satisfies\PAp$ and any $\Sigma_m$ formula $\phi(x)$, there is a $\Sigma_n$-elementary end-extension $N$ for which $N\satisfies\forall x\, \sigma(x)\iff\phi(x)$.
 \end{enumerate}
\end{theorem}

\begin{proof}
For statement (1), let $\rho_k$ be the assertion that $k$ appears on the universal finite sequence enumerated by the algorithm $\tilde e$ of theorem~\ref{Theorem.Universal-algorithm-Sigma-n}, using oracle $0^{(n)}$. This is a $\Sigma^0_{n+1}$ assertion: one asserts that there is a computation, using the correct oracle, showing a stage at which $k$ appears on the sequence. Since the sequence is empty in the standard model, these statements are all false there, and they are possibly necessary with respect to $\uppossible_n$, since once $k$ appears on the sequence, it remains on the sequence in all further $\Sigma_n$-elementary end-extensions. Since any specific number can be added to the sequence, and indeed any coded set of numbers, with no others, these form an independent family of buttons.

For statement (2), let $\sigma_k$ be the assertion that the $k^{th}$ binary digit of the last number enumerated by algorithm $\tilde e$ of theorem~\ref{Theorem.Universal-algorithm-Sigma-n}, using oracle $0^{(n)}$, is one. The complexity of this assertion is $\Sigma_{n+2}$. Since theorem~\ref{Theorem.Universal-algorithm-Sigma-n} shows that we can arrange that this last number is whichever number we want in a $\Sigma_n$-elementary end-extension, we can therefore arrange the pattern of truth for the $\sigma_k$ to be as desired, for any pattern $I$ coded in $M$.

For statement (3), simply adapt the proof of theorem~\ref{Theorem.Flexible-formula}. Let $\sigma(x)=\Phi_m(k,x)$, where $k$ is the final element of the universal $\Sigma_n$-sequence and argue as before, but with $\Sigma_n$-elementary end-extensions.
\end{proof}

One can adapt statement (1) to models of $\PAp$, if one allows the buttons to have parameters, and indeed, the only parameter needed would be the length $u$ of the universal sequence in a given base model of $\PAp$, for then one takes $\rho_n$ as the assertion that $n$ appears on the sequence after $u$. If $\N\satisfies\PAp$, then no parameters are needed, and in any case, no parameters are needed in the switches $\sigma_k$.

The following theorem follows from a classical result due to Wilkie~\cite{Wilkie1975:On-models-of-arithmetic-answers-to-two-problems-raised-by-H-Gaifman}, but is achieved here as a consequence of the universal algorithm.

\begin{theorem}
 For every model $M\satisfies\PAp$, there is a diophantine equation in $M$ with no solutions in $M$, but which does have a solution in some end-extension of $M$ to $N\satisfies\PAp$.
\end{theorem}

\begin{proof}
 Let $n$ be any (possibly nonstandard) number in $M$, such that stage $n$ of the universal algorithm is not successful in $M$. There is an end-extension $N$ of $M$ in which this stage is successful, and this is a $\Sigma_1$ assertion about $n$ that is true in $N$, but not in $M$. This assertion corresponds to a diophantine equation in $M$ with no solution in $M$, but with a solution in $N$.
\end{proof}

One can view the theorem as asserting alternatively that no model of arithmetic has a maximal $\Sigma_1$ diagram: for every model $M\satisfies\PAp$, there are new $\Sigma_1$ assertions about parameters in $M$ that can become true in an end-extension of $M$. Meanwhile, note that the use of parameters in the previous argument is required, since by the argument of theorem~\ref{Theorem.Maximal-Sigma1-extension}, some models of arithmetic have a maximal $\Sigma_1$ theory, as opposed to its $\Sigma_1$ diagram.

A similar argument works for $\Sigma_n$-elementary extensions. The use of parameters in this result also is required, in light of theorem~\ref{Theorem.Maximal-Sigma_n-extension}.

\begin{theorem}
 No model of arithmetic $M\satisfies\PA$ has a maximal $\Sigma_{n+1}$-diagram. For every model $M$ and every natural number $n$, there is some $m\in M$ and a $\Sigma_{n+1}$ statement $\varphi(m)$ that is not true in $M$, but becomes true in a $\Sigma_n$-elementary end-extension of $M$.
\end{theorem}

\begin{proof}
Let $m$ be the first unsuccessful stage of the universal algorithm $\hat e$ of theorem~\ref{Theorem.Universal-algorithm-Sigma-n}, used with oracle $0^{(n)}$ in $M$. Let $\sigma$ be the statement that stage $m$ is successful for this algorithm. This is a statement about $m$ with complexity $\Sigma_{n+1}$, which is not true in $M$, but by the extension property of theorem~\ref{Theorem.Universal-algorithm-Sigma-n}, it becomes true in some $\Sigma_n$-elementary end-extension $N$ of $M$.
\end{proof}

Finally, let me show how Chaitin's incompleteness theorem follows easily from the universal algorithm.

\begin{theorem}
 For any consistent c.e.~theory $T$ extending \PA, there is a finite number $L$ such that for no finite string $s$ does $T$ prove that the Kolmogorov complexity of $s$ exceeds $L$. That is, the theory $T$ cannot prove that any particular string $s$ is strictly harder than $L$ to produce.
\end{theorem}

\begin{proof}
 Let $L$ be the size of the program that runs the first stage of the universal algorithm for the theory $T$ and outputs the result. For any finite sequence $s$, it is consistent with $T$ that this program produces $s$, and so it is consistent with $T$ that the Kolmogorov complexity of $s$ is at most $L$. So $T$ cannot prove that any particular sequence $s$ is harder than this to produce.
\end{proof}

\section{Some background on modal logic and potentialism}\label{Section.Modal-logic-background}

The modal validities of a potentialist system, or indeed any Kripke model of possible worlds, can be fruitfully analyzed by identifying the presence of certain kinds of control statements, such as buttons, switches, ratchets and railyards. A general account is given in~\cite{HamkinsLeibmanLoewe2015:StructuralConnectionsForcingClassAndItsModalLogic}, with numerous examples and applications provided in~\cite{HamkinsLoewe2008:TheModalLogicOfForcing, HamkinsLinnebo2022:Modal-logic-of-set-theoretic-potentialism,HamkinsWoloszyn2024:Modal-model-theory}. I introduce railway switches and the railyard terminology in this article. See also the approach to the modal logic of submodels in~\cite{SavelievShapirovsky2016:On-modal-logic-of-submodels, SavelievShapirovsky2018:On-modal-logics-of-model-theoretic-relations}.

Let me review some of the basics here, following~\cite{HamkinsLinnebo2022:Modal-logic-of-set-theoretic-potentialism}. A \emph{potentialist system} is a Kripke model of first-order structures in a common language, whose accessibility relation refines the substructure relation, so that if world $M$ accesses world $N$, then $M$ is a substructure of $N$. Here, I shall be concerned with the potentialist systems consisting of the models of $\PA$ or of $\PAp$ under the various natural accessibility relations: extension, end-extension, $\Sigma_n$-elementary extension and so forth.

In any potentialist system $\mathcal{M}$, an assertion $\varphi(p_0,\ldots,p_n)$ in the propositional modal language $\mathcal{P}^{\possible}$ (recall the various languages introduced at the end of section \ref{Section.Models-of-arithmetic-as-potentialist-systems}), with propositional variables $p_i$, is \emph{valid} at a world $M$ for a collection $S$ of assertions, if $M\satisfies\varphi(\psi_0,\ldots,\psi_n)$ for every possible substitution $p_i\mapsto\psi_i$ of those propositional variables by assertions $\psi_i\in S$. The validity concept makes sense for assertions $\varphi(p_0,\ldots,p_n)$ in the much broader propositionally-expanded potentialist language $\mathcal{L}^{\possible}(p_0,\ldots,p_n)$, equipping the full potentialist language $\mathcal{L}^{\possible}$ also with propositional variables, although it is common to focus on the propositional modal language $\mathcal{P}^{\possible}$, where the resulting modal theory might be one of the well-known theories.

As I mentioned in the introduction, the modal theory \theoryf{S4} is obtained from the axioms (K) $\necessary(p\implies q)\implies(\necessary p \implies\necessary q)$, (S) $\necessary p\implies p$, (4) $\necessary p\implies \necessary\necessary p$ and (Duality) $\neg\possible p\iff \necessary\neg p$, by closing under necessitation, tautologies and modus ponens. This modal theory, as noted there for the end-extension relation, is easily seen to be valid in every potentialist system with respect to any collection of assertions. It is also easy to see that the {\df converse Barcan} formula $\necessary\forall x\, p\implies\forall x\necessary p$ is valid in all potentialist systems, having substitution instances
 $$\necessary\forall x\, \psi(x)\implies\forall x\necessary \psi(x).$$
The modal theory \theoryf{S4.2} arises by augmenting \theoryf{S4} with the axiom (.2) $\possible\necessary p\to\necessary\possible p$, which is valid in any Kripke model whose accessibility relation is convergent. The modal theory \theoryf{S4.3} arises by including the axiom (.3) $(\possible p\wedge\possible q)\implies\possible[(p\wedge\possible q)\vee(q\wedge\possible p)]$, which is valid in any linearly preordered frame. The modal theory \theoryf{S5} arises by augmenting \theoryf{S4} with the axiom (5) $\possible\necessary p\to p$.

Modal logicians are familiar with the elementary fact that \theoryf{S5} is valid in a Kripke frame $F$, meaning that it is valid in all Kripke models having $F$ as its underlying accessibility relation, if and only if $F$ is an equivalence relation.

But I should like to emphasize that potentialism is not really about this kind of frame semantics, since we are studying natural Kripke models, such as those arising from the models of arithmetic, and they come along with their accessibility relations, such as they are; we shall not be constructing other Kripke models upon those same frames. It can definitely happen that a Kripke model validates \theoryf{S5}, without its frame being an equivalence relation---for example, \theoryf{S5} is valid in any reflexive Kripke model where all worlds have the same atomic propositional truths, regardless of the underlying frame.

For the potentialist systems under consideration in this article, it is generally easy to observe that a given modal theory is valid, in the systems in which it is valid. In this sense, lower bounds are often cheaply found. Meanwhile, what is usually much more difficult is to establish upper bounds on the modal validities of a system. For example, in the main theorem of this article, I will show that the validities of arithmetic potentialism are exactly \theoryf{S4}. It was an easy exercise to show that \theoryf{S4} is valid; what is much harder is to show that no modal assertion outside \theoryf{S4} is valid. And similar issues with the upper bounds arise in essentially all the results of section~\ref{Section.Modal-logic-of-arithmetic-potentialism}.

Regarding upper bounds on the modal validities, a principal advance of my work with Benedikt \Lowe~\cite{HamkinsLoewe2008:TheModalLogicOfForcing}, further developed in~\cite{HamkinsLeibmanLoewe2015:StructuralConnectionsForcingClassAndItsModalLogic} and used in~\cite{HamkinsLinnebo2022:Modal-logic-of-set-theoretic-potentialism}, was the observation that one can often place upper bounds on the modal validities of a system by observing that it admits certain kinds of easily-understood control statements, such as switches, buttons, ratchets and railyards. Because these control statements concern only the worlds in the potentialist system and are stated in the object language $\Larith$, rather than in the modal language, in practice one can often thereby determine the modal validities of a system by using expertise only in the object theory of those structures, rather than expertise in modal logic.

A statement $\sigma$ is a \emph{switch} in a Kripke model $\mathcal{W}$, if $\possible\sigma$ and $\possible\neg\sigma$ are true at every world of $\mathcal{W}$. A collection of switches $\sigma_i$ is \emph{independent}, if every world can access another world realizing any desired finite on/off pattern for those switches.

\begin{theorem}\label{Theorem.Switches-S5}
  If $\mathcal{W}$ is Kripke model and world $W$ admits arbitrarily large finite collections of independent switches, then the propositional modal assertions valid at $W$ are contained in the modal theory \theoryf{S5}, with respect to the language in which the switches are expressed.
\end{theorem}

\begin{proof}
This idea goes back to~\cite{HamkinsLoewe2008:TheModalLogicOfForcing} (see also~\cite{HamkinsLeibmanLoewe2015:StructuralConnectionsForcingClassAndItsModalLogic}), and it is proved explicitly in~\cite[theorem~3]{HamkinsLinnebo2022:Modal-logic-of-set-theoretic-potentialism}. For completeness, let me review the essential idea. If a statement $\varphi(p_0,\ldots,p_n)$ is not part of S5, then it fails in some finite propositional Kripke model $M$ whose underlying frame is the complete relation, where all worlds access all others, and each world $w\in M$ gives truth values to the propositional variables $p_i$.
\begin{figure}[h]\label{Figure.cluster}
\newcommand{\pentacluster}{
    \foreach\t in {0,...,4} {\pgfmathsetmacro\c{25*\t}
                            \draw (18+72*\t:1) node[circle,draw,red!\c!blue] (\t) {};}
    \foreach \r/\s in {0/1,1/2,2/3,3/4,4/0} {\draw[<->,>=stealth] (\r) edge[bend right=20] (\s);}
    \foreach \r/\s in {0/2,2/4,4/1,1/3,3/0} {\draw[<->,>=stealth] (\r) edge[bend right=10] (\s);}
    }
\begin{tikzpicture}
 \pentacluster
\end{tikzpicture}
\caption{A cluster of mutually accessible worlds}
\end{figure}
By duplicating worlds, if necessary, I may assume that there are $2^m$ worlds in $M$. Let $\sigma_i$ be a family of $m$ independent switches in $\mathcal{W}$. For each $t<2^m$, let $\Phi_t$ assert that the pattern of switches conforms with the binary digits of $t$. This provides a partition of the worlds of $\mathcal{W}$ into $2^m$ classes, corresponding to the worlds of $M$. Next, for each propositional variable $p$ appearing in $\varphi$, let $\psi_p=\bigvee\set{\Phi_t\mid (M,t)\satisfies p}$. This statement is true at a world $W$ in $\mathcal{W}$ if and only if $W$ corresponds to a world $t$ in $M$ at which $p$ holds. By induction, one can now establish that
 $$W\satisfies_{\mathcal{W}}\phi(\psi_{p_0},\dots,\psi_{p_n})\qquad\Iff\qquad(M,t)\satisfies\phi(p_0,\dots,p_n),$$
whenever $W$ is a world in $\mathcal{W}$ corresponding to node $t$ in $M$, which is to say that $W\satisfies\Phi_t$. In this way, truth in the propositional Kripke model $M$ is simulated in the Kripke model $\mathcal{W}$. In particular, since $\varphi$ failed at a world of $M$, it follows that $\varphi(\psi_{p_0},\ldots,\psi_{p_n})$ fails at a world of $\mathcal{W}$, and so $\varphi$ is not valid in $\mathcal{W}$. So the modal validities of $\mathcal{W}$ are contained within S5.
\end{proof}

In practice, another convenient way to handle independent switches is with the concept of a \emph{dial}, which is a sequence of statements $d_0,\ldots,d_n$, such that every world satisfies exactly one of the statements and every world can access a world in which any desired one of the dial statements is true. So from any world, you can set the dial to any value that you like, by moving to a suitable accessible world. By considering the binary digits of the dial indices, it is not difficult to see that a Kripke model $\mathcal{W}$ admits arbitrarily large finite collections of independent switches if and only if it admits arbitrarily large finite dials (see~\cite[theorem~4]{HamkinsLinnebo2022:Modal-logic-of-set-theoretic-potentialism}). We therefore conclude:

\begin{theorem}\label{Theorem.Dials-S5}
  If $\mathcal{W}$ is Kripke model and world $W$ admits arbitrarily large dials, then the propositional modal assertions valid at $W$ are contained in the modal theory S5, with respect to the language in which the dials are expressed.
\end{theorem}

Further such connections between modal logics and control statements are provided in~\cite{HamkinsLeibmanLoewe2015:StructuralConnectionsForcingClassAndItsModalLogic}. For example, if a potentialist system admits arbitrarily large families of independent buttons and switches, then the modal validities are contained within \theoryf{S4.2}; if the system admits a long ratchet, then the validities are contained within \theoryf{S4.3}; if it admits arbitrarily large independent families of weak buttons and switches, then the validities are contained within \theoryf{S4.tBA}, the logic of topless Boolean pre-algebras; and so on with other instances.

Let me now introduce and consider here the concept of a \emph{railway switch}, which is a statement $r$ such that $\possible\necessary r$ and $\possible\necessary\neg r$ at a world where it is not yet switched, but becomes switched when $\necessary r$ or $\necessary\neg r$ holds. The train goes one way or the other, but afterwards, it is too late to change tracks.

More generally, a \emph{railyard} is an assemblage of such railway switches. Specifically, if $\mathcal{W}$ is a Kripke model of possible worlds and $T$ is a finite pre-tree (a pre-order whose natural quotient is a tree), such as the one pictured figure~\ref{Figure.Pretree}, then a \emph{railyard} labeling for $T$, or a \emph{$T$-labeling} in the terminology of~\cite{HamkinsLeibmanLoewe2015:StructuralConnectionsForcingClassAndItsModalLogic}, based at world $W$ is an assignment $t\mapsto r_t$ of the nodes of the tree $t\in T$ to assertions $r_t$ in the language of the worlds of $\mathcal{W}$, such that every world of $\mathcal{W}$ satisfies exactly one of the statements $r_t$, the original world $W$ satisfies $r_{t_0}$ for an initial node $t_0$ of the tree, and whenever any $r_t$ is true in a world $U\in\mathcal{W}$, then $U\satisfies\possible r_s$ if and only if $t\leq s$ in $T$. In other words, the assertions $r_t$ partition the worlds of $\mathcal{W}$ in such a way that makes possibility in $\mathcal{W}$ look the same as in $T$. In particular, if the pretree $T$ has nontrivial branching, then each of the labels $r_t$ on the higher nodes will be a railway switch, becoming necessary if you branch into or beyond that cluster and impossible if you branch away from it. In a sense, the finite pre-tree $T$ is realized as a quotient of the Kripke model $\mathcal{W}$, which could have many more worlds or even infinitely many, as it does for the cases in which we are interested.

The existence of sufficient railyard labelings for a potentialist system allows us to conclude that its validities are exactly \theoryf{S4}.

\begin{theorem}\label{Theorem.Pre-tree-labels-give-S4}
 Suppose that $\mathcal{W}$ is a potentialist system that admits a railyard labeling for every finite pre-tree $T$---and it suffices to handle only pre-trees with all clusters the same size and all branching clusters having the same degree---then the potentialist validities of $\mathcal{W}$ are exactly \theoryf{S4}, with respect to the language in which the labeling assertions are expressed.
\end{theorem}

\begin{proof}
Note first that it is easy to see that \theoryf{S4} is valid in any potentialist system, since the accessibility relation is reflexive and transitive (see also~\cite[theorem~2]{HamkinsLinnebo2022:Modal-logic-of-set-theoretic-potentialism}). The converse result, that only \theoryf{S4} is valid, follows from the railyard labelings as an instance of~\cite[lemma~9]{HamkinsLeibmanLoewe2015:StructuralConnectionsForcingClassAndItsModalLogic} and the related general analysis there. Specifically, if a propositional modal assertion $\varphi(p_0,\ldots,p_n)$ with propositional variables $p_i$ is not in \theoryf{S4}, then there is a propositional Kripke model $M$, whose underlying frame is a finite pre-tree $T$, in which $\varphi$ fails, since this collection of frames is complete for \theoryf{S4}. Each world $w\in M$ gives truth values to the propositional variables $p_i$. By duplicating worlds if necessary, we may assume that all the clusters of $T$ have the same size and all branching clusters branch with the same degree. Let $t\mapsto r_t$ be the railyard labeling of $T$ for $\mathcal{W}$. For each propositional variable $p$ appearing in $\varphi$, let $\psi_p=\bigvee\set{r_t\mid (M,t)\satisfies p}$, which is true in $\mathcal{W}$ at exactly the worlds corresponding to nodes in the pre-tree for which $p$ is true in $M$. One may now prove by induction on formulas $\phi$ in propositional modal logic that
 $$W\satisfies_{\mathcal{W}}\phi(\psi_{p_0},\dots,\psi_{p_n})\qquad\Iff\qquad(M,t)\satisfies\phi(p_0,\dots,p_n),$$
whenever $W$ is a world in $\mathcal{W}$ corresponding to node $t$ in $M$, which is to say that $W\satisfies r_t$. In this way, truth in the propositional Kripke model $M$ is simulated in the potentialist system $\mathcal{W}$. In particular, since $\varphi$ failed at a world of $M$, it follows that $\varphi(\psi_{p_0},\ldots,\psi_{p_n})$ fails at a world of $\mathcal{W}$, and so $\varphi$ is not valid in $\mathcal{W}$.
\end{proof}

\section{The modal logic of arithmetic potentialism}\label{Section.Modal-logic-of-arithmetic-potentialism}

I am now finally ready to prove the main results of this article. Recall that $\PAp$ is a fixed consistent c.e.~extension of \PA---perhaps it is \PA\ itself. Let us begin with an easy observation that sets the overall bounds.

\begin{theorem}\label{Theorem.Validities-between-S4-S5}
 In the potentialist system consisting of the models of $\PAp$ under end-extension, the potentialist validities of any model $M$, with respect to assertions in the language of arithmetic $\Larith$ or any extension of it $\Larith^+$, contain \theoryf{S4} and are contained in \theoryf{S5}.
 $$\theoryf{S4}\quad\of\quad\Val_{\uppossible}(M,\Larith^+)\quad\of\quad\Val_{\uppossible}(M,\Larith)\quad\of\quad\theoryf{S5}$$
\end{theorem}

\begin{proof}
The modal theory \theoryf{S4} is valid generally in every potentialist system, as we have noted. For the upper bound, theorem~\ref{Theorem.Infinitely-many-independent-Orey-sentences} provides an infinite family of independent switches, which can be turned on and off so as to realize any desired finite pattern in an end-extension. It follows by theorem~\ref{Theorem.Switches-S5} that the modal validities of any model are contained within \theoryf{S5}.
\end{proof}

\begin{theorem}\label{Theorem.Arithmetic-potentialism-is-S4}
 In the potentialist system consisting of the models of $\PAp$ under end-extension, if $M$ is a model in which the universal algorithm enumerates the empty sequence, then the potentialist validities of $M$, with respect to sentences in the language of arithmetic, are exactly the assertions of \theoryf{S4}.
 $$\theoryf{S4}\quad=\quad\Val_{\uppossible}(M,\Larith)\quad\ofneq\quad\theoryf{S5}$$
\end{theorem}

\begin{proof}
By theorem~\ref{Theorem.Pre-tree-labels-give-S4}, it suffices to provide railyard labelings for any finite pre-tree. Consider such a pre-tree $T$, such as the one pictured below in figure~\ref{Figure.Pretree}. Each world in a cluster can access all the other worlds in that cluster and any world in any higher cluster in the tree.
\begin{figure}[h]
\newcommand{\pentacluster}{\foreach\t in {0,...,4} {\draw (18+72*\t:1) node[circle,draw] (\t) {};}
    \foreach \r/\s in {0/1,1/2,2/3,3/4,4/0} {\draw[<->,>=stealth] (\r) edge[bend right=20] (\s);}
    \foreach \r/\s in {0/2,2/4,4/1,1/3,3/0} {\draw[<->,>=stealth] (\r) edge[bend right=10] (\s);}
    }
\newcommand{\tricluster}{\foreach\t in {0,1,2} {
  \setrandomcolor
    \draw (30+120*\t:.8) node[circle,thick,draw,scale=.6,randomcolor!85!red] (\t) {};}
    \foreach \r/\s in {0/1,1/2,2/0} {\draw[{<[scale=.5]}-{>[scale=.5]},>=Stealth] (\r) edge[bend right=30] (\s);}
    }
\begin{tikzpicture}[scale=.4]
\begin{scope}
  \setrandomcolor
    \draw node[circle,dotted,draw,scale=3,fill=yellow,fill opacity=.1] (root) {};
    \tricluster
\end{scope}
\begin{scope}[shift={(-4,4)}]
  \setrandomcolor
    \draw node[circle,dotted,draw,scale=3,fill=yellow,fill opacity=.1] (left) {};
    \tricluster\end{scope}
\begin{scope}[shift={(4,4)}]
  \setrandomcolor
    \draw node[circle,dotted,draw,scale=3,fill=yellow,fill opacity=.1] (right) {};
    \tricluster\end{scope}
\begin{scope}[shift={(-6,8)}]
  \setrandomcolor
    \draw node[circle,dotted,draw,scale=3,fill=yellow,fill opacity=.1] (leftleft) {};
    \tricluster\end{scope}
\begin{scope}[shift={(-2,8)}]
  \setrandomcolor
    \draw node[circle,dotted,draw,scale=3,fill=yellow,fill opacity=.1] (leftright) {};
    \tricluster\end{scope}
\begin{scope}[shift={(2,8)}]
  \setrandomcolor
    \draw node[circle,dotted,draw,scale=3,fill=yellow,fill opacity=.1] (rightleft) {};
    \tricluster\end{scope}
\begin{scope}[shift={(6,8)}]
  \setrandomcolor
    \draw node[circle,dotted,draw,scale=3,fill=yellow,fill opacity=.1] (rightright) {};
    \tricluster\end{scope}
\draw[->,>=stealth] (root) edge[bend left=10] (left)
 (left) edge[bend left=10] (leftleft)
 (left) edge[bend right=10] (leftright)
 (root) edge[bend right=10] (right)
 (right) edge[bend left=10] (rightleft)
 (right) edge[bend right=10] (rightright);
\end{tikzpicture}
\caption{A finite pre-tree $T$}\label{Figure.Pretree}
\end{figure}

\goodbreak
To produce the railyard labeling, assign each node $t$ in the tree $T$ to an assertion $r_t$ in the language of arithmetic that makes a certain specific claim about the behavior of the universal algorithm $e$. Specifically, let $k$ be such that this tree is at most $k$-branching, and let $m$ be such that all the clusters have size at most $m$. Consider the sequence enumerated by the universal algorithm $e$. It enumerates a finite list of numbers, and from that list we produce the subsequence of numbers consisting of those numbers less than $k$, which we interpret as a way of climbing the tree, a way of successively choosing amongst the branching nodes in $T$, so as to arrive at a particular cluster of $T$ (ignore any additional numbers, if there are too many), plus the last number on the universal sequence that is $k$ or larger (default to $k$, if there is none). By considering this final number modulo $m$, we may interpret it as picking a particular node in the cluster at which we arrived (and simply group together some of the residues if the cluster has fewer than $m$ nodes). In this way, we can assign to each node $t$ of the pre-tree $T$ a statement $r_t$ about the nature of the sequence enumerated by the universal algorithm, in such a way that a world $W$ satisfying $r_t$ will satisfy $\uppossible r_s$ just in case $t\leq s$ in $T$. The reason is that by theorem~\ref{Theorem.Universal-algorithm}, any sequence in any such $W$ can be extended by any desired finite sequence, and we can add to the given sequence as it is computed in $W$ so as to specify in an extension $U$ that we climb to and arrive at any desired node $s$ in the tree. Any model $M$ in which the universal sequence is empty will correspond to an initial node of the tree. Thus, the universal algorithm provides a labeling of any finite pre-tree, and so the potentialist modal validities of $W$, with respect to sentences in the language of set theory, is exactly \theoryf{S4}.
\end{proof}

In particular, if $\PAp$ is true in the standard model $\N$, as in the central case where $\PAp$ is simply \PA\ itself, then the modal validities of $\N$ are exactly \theoryf{S4} for sentences.

\begin{theorem}\label{Theorem.Arithmetic-potentialism-with-parameters}
 In the potentialist system consisting of the models of $\PAp$ under end-extension, the potentialist validities of any model $M\satisfies\PAp$, with respect to assertions in the language of arithmetic allowing parameters from $M$, are exactly the assertions of \theoryf{S4}.
 Indeed, for every model $M$, there is a single parameter $u\in M$, such that the modal validities for arithmetic potentialism over $M$, with respect to assertions in the language of arithmetic using parameter $u$ only, is exactly \theoryf{S4}.
$$\theoryf{S4} \ = \ \Val_{\uppossible}(M,\Larith(M)) \
        = \ \Val_{\uppossible}(M,\Larith(u)) \
        \of \ \Val_{\uppossible}(M,\Larith) \ \of \ \theoryf{S5}.$$
\end{theorem}

\begin{proof}
If parameters are allowed, we can carry out the previous argument in any model of $\PAp$, whether or not the universal sequence is empty in some model. Simply let $u$ be the length of the sequence enumerated by the universal algorithm in $M$, and consider the universal sequence as it might be extended beyond $u$. The idea is that beyond $u$, we again have the tree of possibilities, and we can perform the labeling of finite pre-trees as in theorem~\ref{Theorem.Arithmetic-potentialism-is-S4} by simply ignoring the first $u$ terms of the universal algorithm, leading to a new railyard assignment $t\mapsto r_t$ over $M$, where $r_t$ makes reference to parameter $u$. In this way, we will get $\Val\left(M,\Larith(u)\right)=\theoryf{S4}$, as desired.
\end{proof}

So we've seen how to ensure \theoryf{S4} for sentences in some models of arithmetic, showing that the lower bound of theorem~\ref{Theorem.Validities-between-S4-S5} is sharp. Let me now prove that the upper bound is sharp, by finding a model of $\PAp$ whose potentialist validities are exactly \theoryf{S5}.

A model of arithmetic $M\satisfies\PAp$ satisfies the \emph{arithmetic maximality principle}, if $\uppossible\upnecessary\sigma\to\sigma$ is true in $M$ for every arithmetic sentence $\sigma$. Thus, every sentence that is possibly necessary over $M$ is already true in $M$. In other words, $M$ satisfies the arithmetic maximality principle if and only if \theoryf{S5} is valid in $M$ with respect to sentences in the language of arithmetic. Since theorem~\ref{Theorem.Possibly-necessary-characterization} shows that $\uppossible\upnecessary\sigma$ is equivalent to $\xpossible\xnecessary\sigma$ for arithmetic assertions $\sigma$, the arithmetic maximality principle is equivalently formulated using either end-extensions or arbitrary extensions, and in the next section we will see that it is also equivalently formulated using conservative end-extensions $\consuppossible\consupnecessary\sigma\to\sigma$ or with  computably saturated end-extensions $\soliduppossible\solidupnecessary\sigma\to\sigma$, and others.

Since the arithmetic maximality principle is stated in terms of how a model $M$ relates to its extensions in the potentialist system of all models of arithmetic, it might seem at first that it should be a property of the model, rather than merely of the theory of the model. Nevertheless, the arithmetic maximality principle is revealed by the theory of the model.

\begin{theorem}
 In the potentialist system consisting of the models of $\PAp$ under end-extension, if a model $M$ satisfies the arithmetic maximality principle and $M\eleequiv N$, then $N$ also satisfies the arithmetic maximality principle.
\end{theorem}

\begin{proof}
 If $M$ satisfies the arithmetic maximality principle, then in light of theorem~\ref{Theorem.Possibly-necessary-characterization}, this is visible is the theory of $M$ as follows. Since the maximality principle asserts $\uppossible\upnecessary\sigma\to\sigma$ for every sentence $\sigma$, it is necessary and sufficient that if there is a natural number $n$ such that for all $k$ the statement $\Con(\PAp_k+\neg\Con(\PAp_n+\neg\sigma))$ is in the theory, then $\sigma$ is also in the theory.
\end{proof}

Let me show that indeed there are models of the arithmetic maximality principle. It turns out that these are simply the models satisfying a maximal $\Sigma_1$ theory. A \emph{maximal} consistent $\Sigma_1$ theory over $\PAp$ is a theory $T$ consisting of $\Sigma_1$ sentences only, such that $\PAp+T$ is consistent and $T$ is maximal among such theories. In particular, for such a theory $T$ any $\Sigma_1$ assertion that is consistent with $\PAp+T$ is already an axiom of $T$. It is easy to construct such theories, simply by enumerating all $\Sigma_1$ assertions and then including them one at time, as long as this is consistent over $\PAp$. Indeed, every consistent $\Sigma_1$ theory over $\PAp$ is contained in maximal consistent $\Sigma_1$ theory over $\PAp$.

\goodbreak
\begin{theorem}\label{Theorem.Maximal-Sigma1-extension}
 For any model of arithmetic $M\satisfies\PAp$, the following are equivalent.
  \begin{enumerate}
    \item $M$ fulfills the arithmetic maximality principle.
    \item The $\Sigma_1$ theory of $M$ is a maximal consistent $\Sigma_1$ theory over $\PAp$.
  \end{enumerate}
 These models $M$ have potentialist validities described by
  $$\theoryf{S4}=\Val_{\uppossible}(M,\Larith(M))\ofneq\Val_{\uppossible}(M,\Larith)=\theoryf{S5}$$
\end{theorem}

\begin{proof} Suppose a model of arithmetic $M$ satisfies a maximal $\Sigma_1$ theory over $\PAp$. If $M\satisfies\uppossible\upnecessary\sigma$, then there is an end-extension $N\satisfies\PAp$ satisfying $\upnecessary\sigma$, which means that $N\satisfies\neg\Con(\PAp_k+\neg\sigma)$ for some standard finite $k$ by lemma~\ref{Lemma.Necessity-characterization}. This is a $\Sigma_1$ assertion that is true in $N$, where all the $\Sigma_1$ assertions true in $M$ remain true. By maximality, therefore, this inconsistency statement must already be true in $M$. This implies $\sigma$ is necessary over $M$ and in particular $M\satisfies\sigma$, verifying this instance of the maximality principle.

Conversely, if $M$ satisfies the arithmetic maximality principle, then let me show that the $\Sigma_1$ theory of $M$ is a maximal extension of $\PAp$. Suppose that $\sigma$ is a $\Sigma_1$ assertion that is consistent with $\PAp$ plus the $\Sigma_1$ theory of $M$. By lemma~\ref{Lemma.Possibility-characterization}, this implies $M\satisfies\uppossible\sigma$, and consequently $M\satisfies\uppossible\upnecessary\sigma$, since $\Sigma_1$ assertions, once true, are necessarily true in all further extensions. By the arithmetic maximality principle, therefore, $M\satisfies\sigma$, and so its $\Sigma_1$ theory is maximal.

For the final statement of the theorem, note that we have already established that $\Val(M,\Larith(M))=\theoryf{S4}$ for every model of arithmetic, and the arithmetic maximality principle in $M$ amounts precisely to $\Val(M,\Larith)=\theoryf{S5}$.
\end{proof}

It may seem reasonable to expect that any given model of arithmetic $M$ can be extended so as to achieve a maximal $\Sigma_1$ theory and therefore also the arithmetic maximality principle. Perhaps one might hope, for example, to extend the model one step at a time, making an additional $\Sigma_1$ sentence true each time. Theorem~\ref{Theorem.Extension-with-MP} shows that this expectation is fine, if one seeks only to find an extension of the original model, rather than an end-extension. But meanwhile, corollary~\ref{Corollary.No-end-extension-with-MP} shows that the expectation is wrong for end-extensions: some models of arithmetic have no end-extension to a model with a maximal $\Sigma_1$ theory and hence no end-extension to a model of the arithmetic maximality principle.

\goodbreak
\begin{theorem}\label{Theorem.Extension-with-MP}
 Every model of arithmetic $M\satisfies\PAp$ has an extension (not necessarily an end-extension) to a model of arithmetic $N\satisfies\PAp$ with a maximal $\Sigma_1$ theory, which therefore satisfies the arithmetic maximality principle.
\end{theorem}

\begin{proof}
Consider any model of arithmetic $M\satisfies\PAp$. Let $T_0$ be the theory $\PAp$ plus the atomic diagram of $M$. Enumerate the $\Sigma_1$ sentences as $\sigma_0$, $\sigma_1$, $\sigma_2$ and so on. Build a new theory from $T_0$ by adding $\sigma_n$ at stage $n$, if this is consistent. Let $T$ be the resulting theory. This is consistent, since it was consistent at each stage. So it has a model $N\satisfies T$, which can be taken as an extension of $M$ since $T$ includes the atomic diagram of $M$. To see that the $\Sigma_1$ theory of $N$ is maximal, consider any $\Sigma_1$ sentence $\sigma$, which is $\sigma_n$ for some $n$ and therefore considered at stage $n$ of the theory construction. If $\sigma$ is not true in $N$, then it was not added to the theory $T$, and this must have been because it was inconsistent with $\PAp$ plus the atomic diagram of $M$ plus the earlier part of $T$. But whichever finite part of the atomic diagram that was needed for the inconsistency amounts to a $\Sigma_1$ statement that is true in $M$ and hence also in $N$. So we've proved that any $\Sigma_1$ sentence that is consistent with the $\Sigma_1$ theory of $N$ over $\PAp$ is already true in $N$, as desired.
\end{proof}

\begin{theorem}\label{Theorem.MP-implies-0'}
Every model of the arithmetic maximality principle has the halting problem $0'$ in its standard system.
\end{theorem}

\begin{proof} I shall give two arguments (with thanks to Roman Kossak and Volodya Shavrukov for key suggestions). The first argument relies on a 1978 dissertation result of Lessan (republished in~\cite{Lessan1978Dissertation:Models-of-arithmetic}; see also~\cite{JensenEhrenfeucht1976:Some-problems-in-elementary-arithmetics} and~\cite[section~4]{McAloon1978:Completeness-theorems-incompleteness-theorems-and-models-of-arithmetic}), showing that for a nonstandard model of arithmetic $M\satisfies\PA$, the $\Delta_0$-definable elements are co-initial in the standard cut of $M$ if and only if the $\Pi^0_1$ theory of the standard model is not in the standard system of $M$, or equivalently, if $0'\notin\SSy(M)$. Suppose that $M$ is a nonstandard model of arithmetic and $0'$ is not in the standard system of $M$. We claim that $M$ is not a model of the arithmetic maximality principle. By Lessan's theorem, we know that the $\Delta_0$ definable elements of $M$ are coinitial in the standard cut of $M$. By overspill, we know that $M\satisfies\Con(\PAp_k)$ for some nonstandard $k$, and by making $k$ smaller, if necessary, we may assume that $k$ is $\Delta_0$-definable in $M$ and hence the output in $M$ of some standard finite program $p$ on input $0$; it suffices in this argument for $k$ to be merely $\Sigma_1$-definable. Consider the assertion, ``the number $k$ which is the output of program $p$ on input $0$ satisfies $\neg\Con(\PAp_k)$.'' This is a $\Sigma_1$ sentence, which is not true in $M$, but could become true in some end-extension of $M$, since by the incompleteness theorem we can always end-extend any model of $\PAp$ to make $\neg\Con(\PAp_k)$ for any nonstandard $k$. This is a violation of the arithmetic maximality principle in $M$.

A second, alternative argument proceeds from a theorem of Adamowicz~\cite{Adamowicz1991:On-maximal-theories} (see also~\cite[section 5]{AdamowiczCordonFrancoLaraMartin2015:Existentially-closed-models-in-the-framework-of-arithmetic}), which says that $0'$ is Turing computable from any maximal $\Sigma_1$ theory. If a model of arithmetic $M\satisfies\PAp$ satisfies the arithmetic maximality principle, then its $\Sigma_1$ theory is maximal over $\PAp$ and so $0'$ must be in the standard system of $M$.
\end{proof}

Because the existence of $0'$ in the standard system is expressible in the extension modality $\Larith^{\xpossible}$ by the method of theorem~\ref{Theorem.Modalities-are-different}, perhaps it would be possible to use Lessan's or Adamowicz's methods to answer question~\ref{Question.Up-expressible-from-x?}.

\begin{corollary}\label{Corollary.No-end-extension-with-MP}
 Some models of arithmetic have no end-extension to a model of the arithmetic maximality principle. Indeed, if $0'$ is not in the standard system of $M$, then $M$ has no end-extension to a model of the arithmetic maximality principle.
\end{corollary}

\begin{proof}
To see that some models of arithmetic lack $0'$ in their standard systems, consider the computable tree of attempts to build a complete consistent Henkin theory extending $\PAp$ plus the assertion that a new constant $c$ is infinite (to ensure that the model is nonstandard). It follows by the low basis theorem that there is a low branch, and the Henkin model $M$ arising from such a branch will have an elementary diagram of low complexity. It follows that every set in the standard system of $M$ will be low, and in particular, $0'\notin\SSy(M)$. Since this is preserved to end-extensions, such a model has no end-extension with the arithmetic maximality principle by theorem~\ref{Theorem.MP-implies-0'}.
\end{proof}

Volodya Shavrukov has pointed out in an email message to me that not having an end-extension with the maximality principle is strictly stronger than omitting $0'$ from the standard system, for he constructed a model of arithmetic $M$ whose $\Sigma_1$-definable elements are coinitial in the standard cut, which ensures that there is no end-extension of the maximality principle, yet $0'$ is coded in $M$. This example also shows that one cannot replace $\Delta_0$-definable in Lessan's theorem with $\Sigma_1$-definable.

\begin{question}\label{Question.Characterize-end-extensions-to-MP?}
Can we characterize the models of arithmetic that admit an end-extension to a model of the arithmetic maximality principle? In other words, which models $M$ have an end-extension to a model $N$ with a maximal $\Sigma_1$ theory?
\end{question}

The situation is reminiscent of the maximality principle for forcing (see~\cite{Hamkins2003:MaximalityPrinciple, StaviVaananen2001:ReflectionPrinciples}), which asserts that \theoryf{S5} is valid with respect to the forcing modality for sentences in the language of set theory $\possible\necessary\sigma\to\sigma$, or in other words that every forceably necessary statement $\sigma$ is already true. The similarity is that while \ZFC\ plus this maximality principle is equiconsistent with \ZFC, nevertheless~\cite[theorem 7]{Hamkins2003:MaximalityPrinciple} shows that it is not true that one can always force it over any model of \ZFC. Rather, a model of \ZFC\ has a forcing extension realizing the forcing maximality principle if and only if it has a fully reflecting cardinal $V_\delta\elesub V$.\footnote{Note that this particular issue appears to be missed in~\cite{StaviVaananen2001:ReflectionPrinciples}, where~\cite[theorem 30]{StaviVaananen2001:ReflectionPrinciples} claims, incorrectly, that one can always force the c.c.c.~maximality principle, that is, without using any reflecting cardinal, and similarly in the proofs of~\cite[theorem 27,31,37]{StaviVaananen2001:ReflectionPrinciples}, which define an iteration by inquiring at each stage whether a given statement, of arbitrary complexity, is forceable; this is not possible without a truth predicate. These latter proofs can be repaired by using the reflecting cardinal method as in~\cite{Hamkins2003:MaximalityPrinciple}.} Question~\ref{Question.Characterize-end-extensions-to-MP?} is asking essentially for the arithmetic analogue of this.

Although we have proved that the sentential validities $\Val_{\uppossible}(M,\Larith)$ of a model of arithmetic $M$ are trapped between \theoryf{S4} and \theoryf{S5}, with both of these endpoints being realized, it is not clear exactly which modal theories can be realized.

\begin{question}\label{Question.Intermediate-logics?}
 Which modal theories arise as the collection of sentential validities $\Val_{\uppossible}(M,\Larith)$ for a model of arithmetic $M$? For example, is there a model $M$ realizing exactly \theoryf{S4.2}? or \theoryf{S4.3}? Is this theory always a normal theory? And similarly for $\Val_{\xpossible}(M,\Larith)$ and the other extension modalities considered in this article, can they be intermediate between \theoryf{S4} and \theoryf{S5}?
\end{question}

Benedikt \Lowe\ and I had asked a similar question in~\cite[questions~19,20]{HamkinsLoewe2008:TheModalLogicOfForcing} concerning the modal logic of forcing. Alexander Block and I have answered for that context by providing a model of set theory whose forcing validities are strictly between \theoryf{S4.2} and \theoryf{S5}. The key idea of that argument was to construct a model $M$ with a `last' button, that is, an unpushed button $b$ such that (i) pushing $b$ over $M$ necessarily pushes all other buttons; and (ii) pushing any unpushed button over $M$ pushes $b$. This seems unlikely in arithmetic, in light of the fundamental tree-like structure revealed by the behavior of the universal algorithm.

Volodya Shavrukov has suggested that we might approach question~\ref{Question.Intermediate-logics?} by considering whether there can be a model of arithmetic $M$ whose $\Sigma_1$ theory is not maximal, but which nevertheless satisfies all the sentences that hold in every $\Sigma_1$ maximal extension of $M$. Such a model would validate $\necessary\possible\sigma\to\sigma$ for $\Sigma_1$ sentences $\sigma$, and one could then aim to validate $\necessary\possible\necessary\varphi\to\necessary\varphi$ for all $\varphi$ in the language of arithmetic without validating \theoryf{S5}.

\section{Arithmetic potentialism for other modalities}\label{Section.Other-forms-of-arithmetic-potentialism}

Let me now extend the results of the previous section to the other arithmetic potentialist modalities, such as $\xpossible$, $\uppossible_n$ and $\xpossible_n$, as well as the additional modalities $\soliduppossible$, $\solidxpossible$ and $\consuppossible$, which I shall introduce later in this section.

\subsection{Arbitrary extensions} Let me begin with the modality of arbitrary extensions, defined by
\begin{align*}
  &M\satisfies\xpossible\varphi&\text{ if and only if }\quad &\varphi\text{ holds in some extension of }M,\text{ and} &\\
  &M\satisfies\xnecessary\varphi&\text{ if and only if }\quad &\varphi\text{ holds in all extensions of }M. &
\end{align*}
Much of the analysis of the end-extension modality $\uppossible$ applies to the case of arbitrary extensions $\xpossible$.

\goodbreak\begin{theorem}\label{Theorem.Potentialist-validities-for-extension}
 In the potentialist system consisting of the models of $\PAp$ under the extension modality $\xpossible$, the validities of any model $M\satisfies\PAp$, with respect to arithmetic assertions allowing parameters from $M$, are exactly the assertions of \theoryf{S4}. Indeed, for every $M$ there is $u\in M$, the length of the universal finite sequence in $M$, such that the validities of $M$, with respect to arithmetic assertions using parameter $u$, are exactly \theoryf{S4}.
$$\theoryf{S4} \ = \ \Val_{\xpossible}(M,\Larith(M)) \
        = \ \Val_{\xpossible}(M,\Larith(u)) \
        \of \ \Val_{\xpossible}(M,\Larith) \ \of \ \theoryf{S5}.$$
\end{theorem}

\begin{proof}
 The point is that the railyard labelings provided in the proofs of theorems~\ref{Theorem.Arithmetic-potentialism-is-S4} and~\ref{Theorem.Arithmetic-potentialism-with-parameters} work also with the arbitrary extension modal operator $\xpossible$. Suppose that $M$ is any model of $\PAp$ in which the universal algorithm enumerates the empty sequence, and that $T$ is a finite pre-tree. Let $t\mapsto r_t$ be the railyard labeling of $T$ described in the proofs of theorems~\ref{Theorem.Arithmetic-potentialism-is-S4} and~\ref{Theorem.Arithmetic-potentialism-with-parameters} (so $r_t$ may have parameter $u$, the length of the universal finite sequence in $M$, if this is non-zero). Since the statements $r_t$ were expressed in the language of arithmetic, it follows from lemma~\ref{Lemma.Possibility-characterization} that $\uppossible r_t$ and $\xpossible r_t$ agree, and so this same labeling is a railyard labeling of $T$ for the extension modality just as much as for the end-extension modality.
\end{proof}

The bounds on $\Val_{\xpossible}(M,\Larith)$ provided by theorem~\ref{Theorem.Potentialist-validities-for-extension} are sharp, in the sense that $\Val_{\xpossible}(M,\Larith)=\theoryf{S5}$ in any model of the arithmetic maximality principle, since we have observed that this is equivalently formulated using $\uppossible$ or $\xpossible$; and $\Val_{\xpossible}(M,\Larith)=\theoryf{S4}$ in any model where the universal sequence is empty (or has standard finite length), since then one doesn't need the parameter $u$ as it is absolutely definable.

The argument used in the proof of theorem~\ref{Theorem.Potentialist-validities-for-extension} suggests but does not quite establish that the modal validities for $\uppossible$ and $\xpossible$, with respect to assertions in the language of arithmetic, are the same in any model of arithmetic. So we'd like to ask that now, a question that is closely related to question~\ref{Question.Possibility-equivalence-in-partial-potentialist-language}.

\begin{question}
 If $M\satisfies\PAp$, then is $\Val_{\uppossible}(M,\Larith)=\Val_{\xpossible}(M,\Larith)$? More generally, allowing parameters, is $\Val_{\uppossible}(M,\Larith(A))=\Val_{\xpossible}(M,\Larith(A))$ for any collection of parameters $A\of M$?
\end{question}

Either outcome would be extremely interesting, since either we would have a model whose validities differed between $\uppossible$ and $\xpossible$ or we would have a proof that they must always agree.

\bigskip
\subsection{$\Sigma_n$-elementary extensions} Consider next the extension modalities $\uppossible_n$ and $\xpossible_n$, where we insist that the extensions are also $\Sigma_n$-elementary. Note the special case $\uppossible=\uppossible_0$ and $\xpossible=\xpossible_0$, since every extension is $\Sigma_0$-elementary.

\begin{theorem}
 In either of the potentialist systems consisting of the models of $\PAp$ under $\Sigma_n$-elementary end-extensions $\uppossible_n$ or under arbitrary $\Sigma_n$-elementary extensions $\xpossible_n$, respectively, the potentialist validities of any model $M\satisfies\PAp$, with respect to assertions in the language of arithmetic allowing parameters from $M$, are exactly the assertions of \theoryf{S4}.
 Indeed, for every model $M$, there is a single parameter $u\in M$, such that the potentialist validities over $M$, with respect to assertions in the language of arithmetic using parameter $u$ only, is exactly \theoryf{S4}.
$$\theoryf{S4} \ = \ \Val(M,\Larith(M)) \
        = \ \Val(M,\Larith(u)) \
        \of \ \Val(M,\Larith) \ \of \ \theoryf{S5}.$$
\end{theorem}

\begin{proof}
Consider any finite pre-tree $T$ and let $r\mapsto r_t$ be the railyard labeling of $T$ as in the proof of theorem~\ref{Theorem.Arithmetic-potentialism-is-S4}, except using the oracle universal algorithm $\hat e$ from the proof of theorem~\ref{Theorem.Universal-algorithm-Sigma-n}. The parameter $u$ is simply the length of this universal finite sequence as computed in $M$. The algorithm uses oracle $0^{(n)}$, as defined in whichever model it is employed. Since the content of $0^{(n)}$ is preserved and extended by $\Sigma_n$-elementary extensions, it follows that the previous parts of the computation remain the same as one moves to an accessible world. Thus, it follows just as earlier that this is indeed a railyard labeling of $T$ with respect to $\uppossible_n$ and to $\xpossible_n$, and so the validities for assertions in $\Larith(u)$ are contained within \theoryf{S4} as claimed.
\end{proof}

Consider next the arithmetic $\Sigma_n$-maximality principle, which holds in a model $M\satisfies\PAp$ when the implication $\uppossible_n\upnecessary_n\sigma\implies\sigma$ holds there for every sentence $\sigma$ in the language of arithmetic. The analogues of lemmas~\ref{Lemma.Possibility-characterization} and~\ref{Lemma.Necessity-characterization} hold for $\Sigma_n$-elementary extensions. Namely, a model of arithmetic $M\satisfies\PAp$ satisfies $\uppossible_n\varphi(a)$ for an arithmetic assertion $\varphi$ just in case it satisfies $\Con(\Tr_n+\PAp_k+\varphi(a))$ for every standard finite $k$, and this is also equivalent to $\xpossible_n\varphi(a)$, and the proof is just as in lemma~\ref{Lemma.Possibility-characterization}, except that one includes the $\Sigma_n$ diagram $\Tr_n$ of the model in the theory, to ensure that the extension is $\Sigma_n$-elementary. For this reason, $\uppossible_n\varphi$ is equivalent to $\xpossible_n\varphi$ for arithmetic assertions $\varphi$, and consequently the maximality principle is equivalently formulated either as $\uppossible_n\upnecessary_n\sigma\implies\sigma$, with end-extensions, or as $\xpossible_n\xnecessary_n\sigma\implies\sigma$, with arbitrary extensions.

\begin{theorem}\label{Theorem.Maximal-Sigma_n-extension}
 For any model of arithmetic $M\satisfies\PAp$ and any standard finite natural number $n$, the following are equivalent.
  \begin{enumerate}
    \item $M$ fulfills the arithmetic $\Sigma_n$ maximality principle.
    \item The $\Sigma_{n+1}$ theory of $M$ is a maximal consistent $\Sigma_{n+1}$ theory over $\PAp$.
  \end{enumerate}
 These models $M$ have potentialist validities described by
  $$\theoryf{S4}=\Val_{\uppossible_n}(M,\Larith(M))\ofneq\Val_{\uppossible_n}(M,\Larith)=\theoryf{S5}$$
\end{theorem}

\begin{proof}
Suppose a model of arithmetic $M$ satisfies a maximal $\Sigma_{n+1}$ theory over $\PAp$. If $M\satisfies\uppossible_n\upnecessary_n\sigma$, then there is a $\Sigma_n$-elementary end-extension $N\satisfies\PAp$ satisfying $\upnecessary_n\sigma$, which means that $N\satisfies\neg\Con(\Tr_n+\PAp_k+\neg\sigma)$ for some standard finite $k$ by the $\Sigma_n$-elementary analogue of lemma~\ref{Lemma.Necessity-characterization}. This is a $\Sigma_{n+1}$ assertion that is true in $N$, where all the $\Sigma_{n+1}$ assertions true in $M$ remain true. By maximality, therefore, this inconsistency statement must already be true in $M$. This implies $\sigma$ is $\upnecessary_n$-necessary over $M$ and in particular $M\satisfies\sigma$, verifying this instance of the $\Sigma_n$ maximality principle.

Conversely, if $M$ satisfies the arithmetic $\Sigma_n$ maximality principle, then let me show that the $\Sigma_{n+1}$ theory of $M$ is a maximal extension of $\PAp$. Suppose that $\sigma$ is a $\Sigma_{n+1}$ assertion that is consistent with $\PAp$ plus the $\Sigma_{n+1}$ theory of $M$. By the $\Sigma_n$-elementary analogue of lemma~\ref{Lemma.Possibility-characterization}, this implies $M\satisfies\uppossible_n\sigma$, and consequently $M\satisfies\uppossible_n\upnecessary_n\sigma$, since $\Sigma_{n+1}$ assertions, once true in a $\Sigma_n$-elementary extension, are necessarily true in all $\Sigma_n$-elementary extensions. By the arithmetic $\Sigma_n$ maximality principle, therefore, $M\satisfies\sigma$, and so its $\Sigma_{n+1}$ theory is maximal.

For the final statement of the theorem, note that we have already established that $\Val_{\uppossible_n}(M,\Larith(M))=\theoryf{S4}$ for every model of arithmetic, and the arithmetic $\Sigma_n$ maximality principle in $M$ amounts precisely to $\Val_{\uppossible_n}(M,\Larith)=\theoryf{S5}$.
\end{proof}

\bigskip\goodbreak\subsection{Conservative and saturated extensions}

Let me briefly introduce and analyze the extension modalities $\consuppossible$ and $\soliduppossible$ arising from conservative end-extensions and computably saturated end-extensions, respectively.
\begin{align*}
  &M\satisfies\consuppossible\varphi&\Iff &\quad\varphi\text{ holds in a conservative end-extension of }M,&\\
  &M\satisfies\consupnecessary\varphi&\Iff &\quad\varphi\text{ holds in all conservative end-extensions of }M,&\\
  &M\satisfies\soliduppossible\varphi&\Iff &\quad\varphi\text{ holds in some computably saturated end-extension of }M&\\
  &M\satisfies\solidupnecessary\varphi&\Iff&\quad\varphi\text{ holds in all computably saturated end-extensions of }M.&
\end{align*}
An extension $M\of N$ is \emph{conservative}, if whenever $A\of N$ is definable in $N$ from parameters in $N$, then $A\intersect M$ is definable in $M$ using parameters in $M$. A model $M$ is \emph{computably saturated} (also known formerly as \emph{recursively saturated}), if every computable $1$-type over $M$, with finitely many parameters from $M$, which is consistent with the elementary diagram of $M$, is realized in $M$. We shall usually consider $\soliduppossible$ and $\solidupnecessary$ only when $M$ itself is computably saturated (or the standard model), since otherwise they trivialize. As a memory aid for the symbols, imagine that the conservative end-extension symbol $\consuppossible$ resembles a collared white shirt and tie---conservative attire---and the shaded-in part of $\soliduppossible$ suggests that all computable types over the model are filled-in or realized.

The key observation to make concerning these new modalities is that many of the arguments of the earlier sections of this paper, we built extensions of a model of arithmetic $M$ by taking the Henkin model of a certain theory that was definable inside $M$. All such extensions $N$ arising by this means will be conservative, because any class that is definable in $N$ is a class that $M$ has access to in virtue of its having the entire elementary diagram of $N$. In particular, the universal algorithm result of theorem~\ref{Theorem.Universal-algorithm} works with conservative extensions, and this will enable the key arguments of the potentialist analysis to go through.

\goodbreak
\begin{theorem}For the conservative end-extension modality, with operators $\consuppossible$, $\consupnecessary$:
 \begin{enumerate}
   \item $\consuppossible\varphi(a)$ is equivalent to $\uppossible\varphi(a)$, to $\xpossible\varphi(a)$ and to the other statements of lemma~\ref{Lemma.Possibility-characterization}, for arithmetic assertions $\varphi$.
   \item $\consupnecessary\varphi(a)$ is equivalent to $\upnecessary\varphi(a)$, to $\xnecessary\varphi(a)$ and to the other statements of lemma~\ref{Lemma.Necessity-characterization}, for arithmetic assertions $\varphi$.
   \item The arithmetic maximality principal $\uppossible\upnecessary\sigma\to\sigma$ is equivalently formulated with conservative end-extensions as $\consuppossible\consupnecessary\sigma\to\sigma$.
   \item The universal algorithm of theorem~\ref{Theorem.Universal-algorithm} realizes its extension property with respect to conservative end-extensions.
   \item For any model of arithmetic $M\satisfies\PAp$, the potentialist validities of the conservative end-extension modality $\consuppossible$ obey
       $$\theoryf{S4}=\Val_{\consuppossible}(M,\Larith(M))=\Val_{\consuppossible}(M,\Larith(u))\of\Val_{\consuppossible}(M,\Larith)\of\theoryf{S5}.$$
   \item These bounds on $\Val_{\consuppossible}(M,\Larith)$ are sharp, in the sense that $\Val_{\consuppossible}(M,\Larith)=\theoryf{S4}$ in any model of arithmetic $M\satisfies\PAp$ in which the universal finite sequence is empty or has standard finite length, and $\Val_{\consuppossible}(M,\Larith)=\theoryf{S5}$ in any model $M$ of the arithmetic maximality principle.
 \end{enumerate}
\end{theorem}

\begin{proof}
(1) Clearly $\consuppossible\varphi(a)$ implies $\uppossible\varphi(a)$, because a conservative end-extension is an end-extension. Conversely, if $M\satisfies\uppossible\varphi(a)$, then by the proof of lemma~\ref{Lemma.Possibility-characterization}, one can realize $\varphi(a)$ in an end-extension realized as the Henkin model of a theory definable in $M$, and such an extension will be conservative over $M$. So $M\satisfies\consuppossible\varphi(a)$.

(2) This follows by duality from statement (1).

(3) This follows by the same argument as in theorem~\ref{Theorem.Possibly-necessary-characterization}.

(4) The extension property of the universal algorithm was realized in the Henkin model of definable theory in the model, which is therefore a conservative extension.

(5) This follows by the same argument as in theorems~\ref{Theorem.Arithmetic-potentialism-is-S4} and~\ref{Theorem.Arithmetic-potentialism-with-parameters}, using the fact that the universal algorithm works with conservative extensions. The parameter $u$ here is the length of the universal finite sequence in $M$.

(6) This follows by the same argument as in theorem~\ref{Theorem.Maximal-Sigma1-extension}.
\end{proof}

Regarding the computably saturated end-extension modality $\soliduppossible$, the key observation is that if a model of arithmetic $M$ is computably saturated, then the models $N$ arising as Henkin models for theories in $M$ will also be computably saturated, since their elementary diagram is a definable class in $M$, and any model interpreted in a computably saturated model is itself computably saturated. For this reason, the analysis I have just given for conservative end-extensions $\consuppossible$ carries through almost identically for the potentialist system consisting of the computably saturated models of arithmetic under the end-extension modality $\soliduppossible$, yielding:

\begin{theorem}Concerning the computably saturated end-extension modality, with operators $\soliduppossible$ and $\solidupnecessary$:
 \begin{enumerate}
   \item $\soliduppossible\varphi(a)$ is equivalent, in a computably saturated model, to $\uppossible\varphi(a)$, to $\xpossible\varphi(a)$ and to the other statements of lemma~\ref{Lemma.Possibility-characterization}, for arithmetic assertions $\varphi$.
   \item $\solidupnecessary\varphi(a)$ is equivalent, in a computably saturated model, to $\upnecessary\varphi(a)$, to $\xnecessary\varphi(a)$ and to the other statements of lemma~\ref{Lemma.Necessity-characterization}, for arithmetic assertions $\varphi$.
   \item The arithmetic maximality principal $\uppossible\upnecessary\sigma\to\sigma$ is equivalently formulated, in computably saturated models, with computably saturated end-extensions as $\soliduppossible\solidupnecessary\sigma\to\sigma$.
   \item The universal algorithm of theorem~\ref{Theorem.Universal-algorithm} realizes its extension property, in computably saturated models, with computably saturated end-extensions.
   \item For any computably saturated model of arithmetic $M\satisfies\PAp$, the potentialist validities of the computably saturated end-extension modality $\soliduppossible$ obey
       $$\theoryf{S4}=\Val_{\soliduppossible}(M,\Larith(M))=\Val_{\soliduppossible}(M,\Larith(u))\of\Val_{\soliduppossible}(M,\Larith)\of\theoryf{S5}.$$
   \item These bounds on $\Val_{\soliduppossible}(M,\Larith)$ are sharp, in the sense that $\Val_{\soliduppossible}(M,\Larith)=\theoryf{S4}$ in any computably saturated model of arithmetic $M\satisfies\PAp$ in which the universal finite sequence is empty or has standard finite length, and $\Val_{\soliduppossible}(M,\Larith)=\theoryf{S5}$ in any computably saturated model $M$ of the arithmetic maximality principle.
 \end{enumerate}
\end{theorem}

\goodbreak
A similar analysis can be carried out for the mixed modalities and other modalities, below, and I invite the reader to explore these and other natural extension modalities further.\label{Page.List-of-modalities}
{\relsize{-2}\begin{align*}&M\satisfies\consxpossible\varphi&\Iff &\quad\varphi\text{ holds in some conservative extension of }M&\\
  &M\satisfies\consxnecessary\varphi&\Iff &\quad\varphi\text{ holds in all conservative extensions of }M&\\[1ex]
  &M\satisfies\solidxpossible\varphi&\Iff &\quad\varphi\text{ holds in some computably saturated extension of }M&\\
  &M\satisfies\solidxnecessary\varphi&\Iff&\quad\varphi\text{ holds in all computably saturated extensions of }M&\\[1ex]
  &M\satisfies\solidconsuppossible\varphi&\Iff &\quad\varphi\text{ holds in some c.~saturated conservative end-extension of }M&\\
  &M\satisfies\solidconsupnecessary\varphi&\Iff &\quad\varphi\text{ holds in all c.~saturated conservative end-extensions of }M&\\[1ex]
  &M\satisfies\solidconsxpossible\varphi&\Iff &\quad\varphi\text{ holds in some c.~saturated conservative extension of }M&\\
  &M\satisfies\solidconsxnecessary\varphi&\Iff&\quad\varphi\text{ holds in all c.~saturated conservative extensions of }M&\\[1ex]
  &M\satisfies\ssypossible\varphi&\Iff &\quad\varphi\text{ holds in some extension with same standard system as }M&\\
  &M\satisfies\ssynecessary\varphi&\Iff&\quad\varphi\text{ holds in all extensions with same standard system as }M&\\[1ex]
  &M\satisfies\ipossible\varphi&\Iff &\quad\varphi\text{ holds in some extension preserving a nonstandard cut of }M&\\
  &M\satisfies\inecessary\varphi&\Iff&\quad\varphi\text{ holds in all extensions preserving a nonstandard cut of }M&\\[1ex]
  &M\satisfies\solidssypossible\varphi&\Iff &\quad\varphi\text{ holds in some c.~saturated extension with same standard system}&\\
  &M\satisfies\solidssynecessary\varphi&\Iff&\quad\varphi\text{ holds in all c.~saturated extensions with same standard system}&\\[1ex]
  &M\satisfies\solidipossible\varphi&\Iff &\quad\varphi\text{ holds in some c.~saturated extension preserving a nonstandard cut}&\\
  &M\satisfies\solidinecessary\varphi&\Iff&\quad\varphi\text{ holds in all c.~saturated extensions preserving a nonstandard cut}&
\end{align*}}

Additional modalities arise by further restricting to $\Sigma_n$-elementary extensions, such as $\solidconsuppossible_3\varphi$, which holds in a model $M$ when there is a $\Sigma_3$-elementary conservative computably saturated extension $N$ in which $\varphi$ holds, or $\solidssypossible_2\varphi(a)$, which holds when there is a computably saturated $\Sigma_2$-elementary extension with the same standard system in which $\varphi(a)$ holds. We have a plethora of extension modalities here.

Basically, the situation is that almost all the analysis that I have given for the potentialist validities, the universal algorithm and the maximality principle go through in analogous form for these other extension modalities. Part of what is open is that, although these modalities tend to agree on possibility $\possible\varphi(a)$ for arithmetic assertions $\varphi$, I am not sure, as in question~\ref{Question.Possibility-equivalence-in-partial-potentialist-language}, to what extent the modalities agree on the larger potentialist language ${\possible}\Larith$. Theorem~\ref{Theorem.Modalities-are-different} shows that the validities do not generally agree on the full potentialist language $\Larith^{\possible}$, and that argument will extend to several instances for these further mixed modalities, but perhaps some of them do agree there.

It is also not clear for a given model $M$ and $A\of M$ whether the validities $\Val_{\possible}(M,\Larith(A))$ are the same for all the various modalities (assume $M$ is computably saturated when considering the computably saturated modalities). Of course, in some models they are, such as in the models of the arithmetic maximality principle, or when $A=M$, since in these cases we get \theoryf{S5} and \theoryf{S4}, respectively. But we are lacking a fully general argument that these validities are always the same for the various modalities.

\begin{question}
 Exactly to what extent are the potentialist validities of these various arithmetic extension modalities the same?
\end{question}

This is an open-ended question, which would seem to admit many partial results by considering merely some of the modalities. I invite the reader to join the project and help sort out these various matters.

Volodya Shavrukov has mentioned that his article~\cite{Shavrukov2016:Duality-non-standard-elements-and-dynamic-properties-of-re-sets}, although it does not mention modal logic explicitly, can be seen as concerned with the arithmetic potentialism of nonstandard models of true arithmetic under arbitrary extensions preserving a distinguished nonstandard element.

\section{Concluding philosophical remarks}\label{Section.Philosophical-remarks}

Let me now finally return to the philosophical issues mentioned at the outset. I take this paper to illustrate the pattern of exchange between philosophy and mathematics that I had highlighted, by which a philosophical idea inspires a mathematical analysis, which in turn raises further philosophical issues, and so on in a fruitful cycle. The philosophy of potentialism originates in antiquity in the classical dispute between actual and potential infinity, and current work spans the range from philosophy to mathematics and back again several times around. Having now explored in mathematical detail different kinds of arithmetic potentialism, validating in some cases fundamentally different modal theories, we might recognize that in order for a philosophical account of potentialism to be satisfactory, it must address how it is situated with respect to the key points of contention. 

To press the philosophical discussion further, I should like to argue more specifically that the convergent forms of potentialism are far closer to actualism than are the more radically potentialist theories---I regard convergent potentialism as implicitly actualist. The reason is that if the universe fragments of one's potentialist system are mutually coherent with one another, forming a coherent system in the sense of~\cite{HamkinsLinnebo2022:Modal-logic-of-set-theoretic-potentialism}, then there is a unique limit model to which the system is converging, and there seems to be very little at stake in the ontological dispute concerning the actual existence of this limit model. What does it matter if the potential objects that might come to exist do not yet actually exist, if the way that they will come to exist is unique and deterministic? To experience potentiality in convergent potentialism is simply to wait for the inevitable. In a convergent potentialist system, the potentialist can accurately refer to truth in the limit model, without officially having that model in his or her ontology, by means of the potentialist translation: one translates existence assertions $\exists x$ for the limit model as $\possible\exists x$ for the universe fragments (see~\cite[theorem 1]{HamkinsLinnebo2022:Modal-logic-of-set-theoretic-potentialism} and also~\cite{Linnebo:2013-PHS}; a similar idea underlies theorem~\ref{Theorem.Projective-truth-is-expressible} in this article). For convergent potentialism, therefore, it is as though the limit model actually exists, for all the purposes of speaking about what is true or false there. But it is not just about making truth and falsity assertions for the limit model; rather, with convergent potentialism, it is that the limit model has an implicit existence whose fundamental nature is determined by and definable in the potentialist system, for we can interpret the full actualist universe inside the potentialist ontology. This is the sense in which I regard convergent potentialism as implicitly actualist. In convergent potentialism, the full actual limit universe supervenes on the potentialist ontology.

A more radical form of potentialism, in contrast, arises when there is truly branching possibility, statements that could become true, but might not. For this kind of potentialism, one is living in a universe fragment, and the question of what might become true and verified in a larger universe fragment depends on precisely how the universe unfolds. Will the next number on the universal finite sequence be even or odd? Will the Rosser sentence be true or its negation? It depends on which possibilities will become actual.

Woodin~\cite{Woodin2011:A-potential-subtlety-concerning-the-distinction-between-determinism-and-nondeterminism} uses this feature explicitly to make a point about free will and determinism. Suppose I fix the universal algorithm of theorem~\ref{Theorem.Universal-algorithm} and announce that it will predict your free-will choices. Am I wrong? You freely choose a finite sequence of numbers, any sequence at all, and we run the program. If time is extended into a suitable universe, my program will correctly enumerate your sequence there, fulfilling my claim of it as a predictor, just like Newcomb's perfect predictor. And we can do it again from that point, as much as we want: you freely pick a finite extension of the sequence, and in the right extension of the universe, the predictor will again be right. Woodin's point in part is that seemingly free-will choices can be seen ultimately as deterministic, agreeing perfectly with the output of a predictor algorithm that is fixed in advance. The trick is that the metaphysical context in which that algorithm is to be run must be chosen carefully by carefully extending the universe in a suitable way so as to achieve the accurate prediction. In this way, the argument shows how discussions of free-will and determinism become wrapped up with metaphysical questions concerning the criterion by which we determine which universe it is that we take ourselves to inhabit, especially when this universe may have many distinct proper extensions.

I believe that end-extensional arithmetic potentialism can shed light on the philosophy of finitism and even ultrafinitism. To be sure, I do not expect to engage the ultrafinitist in an analysis of nonstandard models of arithmetic, since such models and even the standard model of arithmetic do not have a real mathematical existence for the ultrafinitist. Rather, what I am proposing is to use the potentialist system in order to shed light on what are the commitments of ultrafinitism, for it is often not as clear as one might hope to come to an understanding of exactly what ultrafinitism is. This is an analysis of ultrafinitism undertaken by and for the actualists, to help them understand ultrafinitism, much as one might use classical logic when analyzing or comparing the power of various systems of intuitionist logic.

I find this approach helpful because important features of the potentialist system appear to be shared with assertions that one sometimes hears from ultrafinitists. For the ultrafinitist, the universe of natural numbers starts out perfectly clear with the numbers $0$, $1$, $2$ and so on, but as time proceeds there is increasing hesitancy concerning extremely large numbers; it is as though the numbers get less definite as one proceeds; one can often describe much larger numbers with a comparatively small definition, but the number being defined would be so vast that the ultrafinitist is hesitant to agree that it actually exists. Thus, the ultrafinitist appears to have something like an initial segment of the universe of natural numbers, a realm of feasibility. Perhaps this realm is even closed under successor, or perhaps not, but many ultrafinitists are reluctant to assert that there is a largest natural number, and it is because of this kind of issue that it is often difficult to say exactly what it is that the ultrafinitist holds. (See further discussion in \cite{Hamkins2025:Potentialist-conception-of-ultrafinitism}.)

My idea is that every model of arithmetic $M$ can be seen as providing an ultrafinitistic context, a realm of feasibility, with respect to its end-extensions $N$, in that the objects of $M$ are smaller in a very robust way than the additional objects of $N$, yet still they obey in $M$ the attractive and familiar mathematical properties. In particular, since $M$ is closed under successor and the other elementary arithmetic operations (but definitely not necessarily under all arithmetically definable operations in $N$), what we gain is a robust, coherent and mathematically precise way of understanding the nature of ultrafinitist worlds. On this account, the finitist and ultrafinitist perspective is that the full entirety of the natural numbers is so vast that it has these difficult-to-describe cuts corresponding to partial universe fragments, which are closed under successor and much more, but which can be viewed as a lower realm of feasibility. Thus, one doesn't view the nonstandard models as enlargements of the ``standard'' model, but rather one views them as ultrafinitist approximations of the full model of natural numbers yet to be constructed on top.

To be clear, the view is not that the finitist or the ultrafinitist are speaking about nonstandard models of \PA, which of course they are not. Rather, the view is that the situation arising from the potentialist account of these models seems to exhibit many of the features concerning arithmetic and the nature of arithmetic truth and number existence assertions that the finitist and the ultrafinitist seem to describe, and for this reason by studying the nature of the potentialist system, we might come to a better understanding of those views.

Thus, I propose to view the philosophy of ultrafinitism in modal terms as a form of potentialism. For this kind of ultrafinitism, we would have a hierarchy of realms of feasibility, about which we could make modal assertions connecting them. Further, this picture leads directly to the question of branching possibility in ultrafinitism, and we would thus seem to have distinct varieties of ultrafinitism, depending on how one answers. Is it part of the ultrafinitist ontology or not that the nature of the numbers, as we produce more and more of them, are determinate with linear inevitability? Or might we discover different arithmetic truths if the numbers are revealed differently? Linear inevitability ultrafinitism would have modal validities \theoryf{S4.3}, but would, as I argued above, be much closer to non-ultrafinitist positions concerning the limit model. Radical branching ultrafinitism, in contrast, would exhibit true branching possibility as the realms of feasibility unfold, validating only~\theoryf{S4}.

The potentialist version of ultrafinitism need not necessarily be arithmetic end-extensional potentialism, for one can imagine that the ultrafinitist realms of feasibility do not necessarily constitute an initial segment of the larger realms. For example, it is conceivable that a modal ultrafinitist could hold that $2^{2^{100}}$ comes into existence at an earlier stage of feasibility than some of the numbers smaller than it, simply because it is easier to describe this number---it has low Kolmogorov complexity---than some of the numbers smaller than it, which can have enormous complexity or an effectively random nature, say, for their digits.

Turning now to another philosophical topic, logically inclined mathematicians sometimes inquire: if a theory is inconsistent, but there is no short proof of a contradiction from the theory, can one still rely on it when using only short proofs? The potentialist framework provides an interesting take on this, for one can imagine that a theory is consistent in a universe fragment, but potentially inconsistent only in a larger fragment. Indeed, this phenomenon is fundamental to the proof of the universal algorithm, since if $n$ is the last successful stage of that algorithm, then $\PA_{k_n-1}$ is consistent in that model, but it can become inconsistent in an end-extension. The smaller world can usefully build a model of a theory, even though the larger model cannot really do this sensibly, when the theory becomes inconsistent. In this sense, the smaller universe fragments actually have access to a broader number of possible worlds; they haven't yet been closed off by selecting a particular universe extension. Such a perspective allows one to look upon paraconsistency through a potentialist lens, which provides in effect a hierarchy of realms of consistency and thereby controls logical explosion.

Finally, let me briefly discuss how arithmetic potentialism relates to the interpretation of Maddy's \textsc{maximize} principle~\cite{Maddy1997:NaturalismInMathematics} in arithmetic foundations. Maddy directs us by her maxim to adopt foundational theories that maximize the range of mathematical possibility. She speaks of maximizing the space of available isomorphism types, and a naive interpretation of that would seem to lead us always to prefer end-extending the model of arithmetic, since every extension realizes additional isomorphism types, which are not available in the previous model. No model can be fully maximal in this sense, however, since every model has proper end-extensions. Let me propose an alternative natural interpretation of the maximize imperative, however, which is that we should maximize the collection of sentences that are necessarily true. This would be true in the models of the arithmetic maximality principle, where every sentence $\sigma$ that could become necessarily true is already necessarily true. By theorem~\ref{Theorem.Maximal-Sigma1-extension}, these are precisely the models exhibiting a maximal $\Sigma_1$ arithmetic theory---we would be maximizing the true existence assertions in arithmetic, which is surely a way of fulfilling her idea about maximizing isomorphism types. The problem for Maddy, however, is that all such theories necessarily make numerous inconsistency assertions. In particular, they all think $\neg\Con(\PA)$ and much more. Since they have maximized the collection of sentences $\sigma$ for which $\necessary\sigma$ holds, they have also maximized the collection of true sentences $\sigma$ for which $\PA_k+\neg\sigma$ is inconsistent. But Maddy doesn't usually seem to take her \textsc{maximize} principle to compel one towards holding that most mathematical theories are inconsistent.

Maddy, of course, is implementing her principle in set theory, rather than arithmetic, and she proposes it as a principled way to justify various large cardinal axioms. The point I make here is that adding new large cardinal axioms in set theory leads usually to the \emph{negation} of $\Sigma^0_1$ statements, since they cause consistency statements to become true rather than inconsistency statements. So perhaps the large-cardinal set-theorist reply is that we should be \emph{minimizing} the $\Sigma^0_1$ theory, or in other words, maximizing the $\Pi^0_1$ theory. This makes sense if one thinks that having a shorter model of arithmetic is better: it is closer to being standard. The second-order induction axiom, after all, asserts that the natural numbers are minimal with respect to containing $0$ and being closed under successor.

\printbibliography

\end{document}